\documentclass[USenglish,12pt]{amsart}	
\usepackage{ytableau, underscore}
\usepackage{quiver, cancel}

\usepackage{microtype}
\usepackage{booktabs}
 
\usepackage{hyperref, amssymb, amsmath, amsthm, color}
\usepackage{geometry}
\usepackage[T1]{fontenc}
\usepackage{textcomp}

\usepackage[maxcitenames=4, mincitenames=4,
maxbibnames=99, minbibnames=99]{biblatex}
\makeatletter
\renewcommand{\tocsection}[3]{%
  \indentlabel{\@ifnotempty{#2}{\bfseries\ignorespaces#1 #2\quad}}\bfseries#3}
\renewcommand{\tocsubsection}[3]{%
  \indentlabel{\@ifnotempty{#2}{\ignorespaces#1 #2\quad}}#3}

\newcommand\@dotsep{4.5}
\def\@tocline#1#2#3#4#5#6#7{\relax
  \ifnum #1>\c@tocdepth 
  \else
    \par \addpenalty\@secpenalty\addvspace{#2}%
    \begingroup \hyphenpenalty\@M
    \@ifempty{#4}{%
      \@tempdima\csname r@tocindent\number#1\endcsname\relax
    }{%
      \@tempdima#4\relax
    }%
    \parindent\z@ \leftskip#3\relax \advance\leftskip\@tempdima\relax
    \rightskip\@pnumwidth plus1em \parfillskip-\@pnumwidth
    #5\leavevmode\hskip-\@tempdima{#6}\nobreak
    \leaders\hbox{$\m@th\mkern \@dotsep mu\hbox{.}\mkern \@dotsep mu$}\hfill
    \nobreak
    \hbox to\@pnumwidth{\@tocpagenum{\ifnum#1=1\bfseries\fi#7}}\par
    \nobreak
    \endgroup
  \fi}
\AtBeginDocument{%
\expandafter\renewcommand\csname r@tocindent0\endcsname{0pt}
}
\def\l@subsection{\@tocline{2}{0pt}{2.5pc}{5pc}{}}
\makeatother

\newtheorem{thm}{Theorem}

\newtheorem{cor}{Corollary}

\theoremstyle{definition}
\newtheorem{defi}{Definition}
\newtheorem{example}{Example}
\newtheorem{rmk}{Remark}

\DeclareMathOperator{\Sym}{Sym}
\newcommand{\BB}[2]{\boldsymbol{B}_{#1,#2}}
\newcommand{\rBB}[3]{\widehat{\boldsymbol{B}}_{#1,#2,#3}}
\newcommand{\ulmbd}{\underline{\lambda}}

\title{Two Formulas for the Number of Lines on Complex Projective Hypersurfaces}

\author{Javier Álvarez-Vizoso}
\address{Max Planck Institute for Gravitational Physics (Albert Einstein Institute),\newline
Callinstr. 38, 30167 Hannover, Germany}
\email{javizoso@alumni.colostate.edu}
\date{}

\begin{document}

\begin{abstract}
Two formulas for the classical number $C_n$ of lines on a generic hypersurface of degree $2n-3$ in $\mathbb{CP}^n$ are obtained which differ from the formulas by Dominici, Harris, Libgober, and van der Waerden-Zagier. We review the splitting principle computation by Harris obtaining a similar general closed-form formula in terms of the Catalan numbers and elementary symmetric polynomials. This in turn yields $C_n$ as a linear difference recursion relation of unbounded order. Thus, for the sequence of certain linear combinations of $C_n$, a simple generating function is found. Then, a result from random algebraic geometry by Basu, Lerario, Lundberg, and Peterson, that expresses these classical enumerative invariants as proportional to the Bombieri norm of particular polynomial determinants, yields another combinatorial expansion in terms of certain set compositions and block labeling counting. As an example, we compute this combinatorial interpretation for the cases of 27 lines on a cubic surface and 2875 lines on a quintic threefold. As an application, we reobtain the parity and asymptotic upper bound of the sequence. In an appendix, we generalize the splitting principle calculation to obtain a formula for the number of lines on a generic complete intersection.
\end{abstract}

\maketitle

\Small
\tableofcontents
\normalsize


\section{Introduction}

Enumerative geometry \cite{32642016} deals with the problem of counting the finite number of solutions to special geometric configurations, usually posed in terms of intersection theory and moduli spaces. 
By a dimensional argument on intersections in the moduli space of projective lines, it is expected that the number of lines contained on a generic complex hypersurface of degree $2n-3$ in projective space $\mathbb{CP}^n$ is a finite number. This is because the Grassmannian variety parametrizing lines in $\mathbb{CP}^n$ has dimension $2(n-1)$ and there are $d+1$ conditions for a line to be contained in a general hypersurface, so when $d+1=2(n-1)$ a zero-dimensional intersection is expected, i.e., a finite number of points in the moduli space of lines. Obtaining this sequence of numbers can be regarded as one of the major driving forces in the development of enumerative geometry. It is a classical result by A. Cayley and G. Salmon in 1849 that there are 27 lines on a cubic surface in projective three-dimensional space. The number 2875 for the corresponding case of the quintic three-fold in $\mathbb{CP}^4$ has a long history, starting with H.  Schubert's computation in 1879, and recently connected to theoretical physics via the prediction of this and other invariants, such as the number of rational curves, from mirror symmetry \cite{CANDELAS199121} and quantum cohomology \cite{Kontsevich95,Kontsevich1994}. These enumerative invariants define the sequence $C_n$ for $n\geq 2$:
\begin{equation}
 \{ 1, 27, 2875, 698005, 305093061, 210480374951, 210776836330775,\dots\}.
\end{equation}


The classical computation is carried out intersecting Schubert cycles, and the modern form involves employing the splitting principle to calculate the degree of the top Chern class (Euler class) of the $(2n-3)$-th symmetric product of the dual universal subbundle $S^*$ on the Grassmannian of projective lines $\mathbb{G}r(1,n)$, see \cite[ch. 6]{32642016}:
\begin{equation}\label{eq:Chern}
	C_n = \int_{\mathbb{G}r(1,n)} c_{2n-2}(\Sym^{2n-3}(S^\ast)).
\end{equation}
By using a representation of the Chow ring of $\mathbb{G}r(1,n)$ with Gröbner bases, this can be automatized and computational algebra packages such as \texttt{Schubert2} for the software \texttt{Macaulay2} can readily produce the values of the sequence by symbolic calculations.   

Using Atiyah-Bott localization, it can be shown that the integral of the top Chern class is a rational function of arbitrary complex variables that produces the sequence of numbers. D. Zagier \cite{grunberg2006arxiv} has proved that this is indeed independent of the variables by finding an alternative expression of the sequence as the coefficients of $x^{n-1}$ in a polynomial product in one variable:
\begin{equation}\label{eq:Zagier1}
	C_n = \left[ (1-x)\prod_{k=0}^{2n-3}(2n-3-k-kx) \right]_{x^{n-1}}
\end{equation}
which was already proved using classical Schubert calculus by B. L. van der Waerden \cite{Waerden33} in 1933.
This leads to one of the simplest closed-form formulas known for the sequence, in terms of the unsigned Stirling numbers of the first kind:
\begin{equation}\label{eq:Zagier2}
	C_n = \sum_{m=0}^{n-1}(-1)^{n-1-m}\binom{2n-2-m}{n-1}(2n-3)^{m+1}{ 2n-3\brack m}.
\end{equation}
D. Grünberg and P. Moree \cite{Grnberg2006SequencesOE} used this formula to study the arithmetic congruence properties of these invariants, e.g., they are all odd integers, as well as their asymptotic behavior at large $n$. However, much remains unknown about the sequence in terms of recurrence relations or generating functions.

Another closed-form formula had been discovered by A.S. Libgober in 1973, \cite{Libgober1973}, using intersection theory on the Grassmannian to determine the class of the Fano variety of lines as a combination of Schubert cycles without employing Chern classes:
\begin{equation}\label{eq:libgober}
C_n =	(2n-3)(2n-3)![e_{n-2}(L_{2n-4})- e_{n-3}(L_{2n-4})] \text{ with } L_{2n-4}=\left\{\frac{2n-3-j}{j}\right\}_{j=1}^{2n-4},
\end{equation}
where the $e_t(\cdot)$ are elementary symmetric polynomials. Libgober establishes a more general formula for the number of lines on generic complete intersections of any fixed codimension, which is again related to mirror symmetry by being the first term of the sequence of the expected number of rational curves, \cite{Libgober1993}.
P. Dominici states a very similar formula in \cite{Dominici98}, (see \cite[Th. 4.26]{Grnberg2006SequencesOE} for a proof):
\begin{equation}\label{eq:Dominici} 
	C_n = (2n-3)^2(2n-4)![ e_{n-2}(Y_{2n-4}) - e_{n-1}(Y_{2n-4}) ], \text{ with } Y_{2n-4}=\left\{\frac{j}{2n-3-j}\right\}_{j=1}^{2n-4}.
\end{equation}
For a general review of results on the Fano scheme of lines see \cite{ciliberto2020lines}.

Finally, using the splitting principle with Chern classes, J. Harris obtains in 1979, \cite{Harris79}, the following expression in terms of the Catalan numbers:
\begin{equation}\label{eq:Harris}
C_n = (2n-3)(2n-3)!\sum_{k=0}^{n-3}\frac{(2k)!}{k!(k+1)!}\sum_{\substack{I\subset\{1,2,\dots,n-2\}\\ |I|=n-2-k }}\prod_{i\in I}\frac{(2n-3-2i)^2}{i(2n-3-i)}.
\end{equation}

The first aim of the present work is to review this general calculation in Schubert calculus using the splitting principle with Chern classes for arbitrary $n$. We thus obtain a similar formula that yields the number of lines on hypersurfaces in terms of the Catalan numbers and the elementary symmetric polynomials evaluated at the numbers $\left\{\frac{(2n-3-j)j}{(2n-3)^2}\right\}_{j=1}^{n-2}$. This in turn yields the numbers $C_n$ as an non-homogeneous linear combination of their predecessors, $C_m, 2\leq m\leq n-1$, which is sometimes called linear difference equation of unbounded order, \cite{MALLIK199879}, and can be regarded as a generalized linear recurrence relation with variable coefficients. Inverting the linear dependencies, one obtains a sequence of linear combinations of $C_n$ which has a simple generating function with closed-form formula, related to the one generating the Catalan numbers via a generalized binomial transformation. Equating this expression, or Harris', to the one by van der Waerden-Zagier trivially establishes a relationship between the Catalan numbers and the unsigned Stirling numbers of the first kind. In an appendix, we generalize this calculation to obtain a formula for the number of lines on a generic complete intersection of arbitrary fixed codimension.

Our second result is a new formula that offers an alternative combinatorial expansion in terms of factors summing over set compositions and block labelling counting. It is deduced from random algebraic geometry, a subject initiated by Edelman, Kostlan, Shub and Smale \cite{edelman1995many, edelman1994many,Kostlan2002OnTE,Shub1993b,Shub1993a,Shub1993c}, and developed, e.g., by Bürgisser and Lerario \cite{BrgisserLerario2020}. It uses probabilistic methods to study geometric problems that can be characterized by random polynomials. In particular, S. Basu, A. Lerario, E. Lundberg, and C. Peterson \cite{Basu2019} find a fundamental formula for the number of lines that shows it is proportional to the statistical expected value of the absolute value of certain determinants arising from random matrices. In the complex case, this expected value can be interpreted as the Bombieri norm of such determinants. We directly calculate the Bombieri norm of these polynomials for arbitrary dimension and interpret the result in terms of the combinatorics of certain blocks and set compositions. 

As an example, we explicitly compute the combinatorial expansion for the 27 lines on a cubic and 2875 lines on a quintic, using only the combinatorics of block cycle labelling and composition counting. As an application, besides the linear difference recurrence relations and the generating function of the linear combination of $C_n$, which are new results, we straightforwardly reobtain the known facts that the $C_n$ are odd numbers, and get an upper bound of the same order as the known asymptotic limit $\log C_n\sim 2n\log(n)+O(n),\;\text{ for }n\rightarrow\infty,$ using elementary methods. These properties, and many more, are known from van der Waerden-Zagier expression using congruences and more sophisticated techniques, like residue calculus on the integral representation of $C_n$ or the $\text{asymp}_k$ trick, cf. \cite{grunberg2006arxiv,Grnberg2006SequencesOE}.
Six closed-form formulas thus exist which may provide complementary approaches to derive results about the structure of this important enumerative geometry sequence, from arithmetics (congruences), and combinatorics (recursion relations), to analysis (generating functions).

The paper is organized as follows: in Section \ref{sec:main} we introduce the notation and state the two main theorems, giving examples of the combinatorial objects involved and deriving the parity and asymptotic upper bound of the sequence. Also, the variable coefficient recursion for $C_n$ is stated and the generating function of linear combinations of $C_n$ is derived. In Section \ref{sec:examples} we explicitly compute the combinatorial expansion for the cases of cubic surfaces and quintic threefolds, which reflect the counting of certain monomials in determinants arising in random algebraic geometry. 
In Section \ref{sec:Schubert} we prove the Schubert calculus formula using intersection theory via Chern classes, the splitting principle for arbitrary dimension, and basic properties of the Schubert cycles involved.
In Section \ref{sec:Bombieri} we prove the combinatorial expansion formula arising from random algebraic geometry by explicitly computing the Bombieri norm of certain polynomial determinants that arise in the Basu-Lerario-Lundberg-Peterson theorem, where we see that the monomial counting involved is directly related to the combinatorics of certain set compositions and blocks labeling. 
Finally, we conclude with a brief discussion of the results. In addition, in Appendix \ref{sec:append} we generalize the Schubert calculus computation with the splitting principle to obtain a formula for the number of lines on complete intersections.
%
%
\section{Main statements}\label{sec:main}


\subsection{Schubert calculus formula}

Let $S^\ast$ be the dual to the universal subbundle $S$ of the Grassmannian $Gr(2,n+1)=\mathbb{G}r(1,n)$ of projective lines in $\mathbb{CP}^n$. The universal subbundle has as fiber at a point $[\Lambda]\in G(2,n+1)$ the space $\Lambda$ itself, and from its dual the equation of a line can be given as a section of $\Gamma(\mathbb{G}r(1,n),S^\ast)$. 
The defining equation of a hypersurface of degree $2n-3$, being a homogeneous polynomial, can be specified as a generic section $s\in\Gamma(\mathbb{G}r(1,n),\Sym^{2n-3}(S^\ast))$. Thus, a line $L\in\mathbb{CP}^n$ is a subvariety of the hypersurface if and only if the section $s$ vanishes on $[L]\in\mathbb{G}r(1,n)$. The number of points of this expected zero-dimensional scheme is precisely the degree of the top Chern class of the bundle, the Euler class, i.e., $\deg c_{2n-2}(\Sym^{2n-3}(S^\ast))$, see \cite[ch. 6]{32642016}.

Our first result aims to review this modern Schubert calculus computation using the splitting principle, as already done by Harris in \cite{Harris79}, by obtaining a formula similar to Equation \eqref{eq:Harris}.
\begin{thm}\label{th:SchubertCn}
	The number $C_n$ of lines on a generic complex projective hypersurface of degree $2n-3$ in $
	\mathbb{CP}^n$ is 
\begin{equation}\label{eq:SchubertCn}
\boxed{
	C_n = (2n-3)^{2n-2}\sum_{m=0}^{n-2}e_m(\Gamma_{n-2})(-1)^m(2m+1)K_m 
	}
\end{equation}
where $K_m=\binom{2m}{m}\frac{1}{m+1}$ are the Catalan numbers, and $e_m(\Gamma_{n-2})$ are the elementary symmetric polynomials (with $e_0=1$) evaluated at the set of numbers $\Gamma_{n-2}=\{\frac{(2n-3-j)j}{(2n-3)^2}\}_{j=1}^{n-2}$.
\end{thm}
By equating this to van der Waerden-Zagier's formula \eqref{eq:Zagier2}, one trivially obtains a relationship between the Catalan numbers and the unsigned Stirling numbers of the first kind.

The sum over $m$ in the previous equation can be read as an inner product between a row of an infinite lower triangular matrix and an infinite vector, for which sequentially solving leads to an expression of every number $C_n$ in terms of its predecessors in the sequence. We shall prove the following two corollaries at the end of the main proof in Section \ref{sec:Schubert}.

\begin{cor}\label{cor:recursion}
Every number $C_n$ is given in terms of its predecessors in the sequence as:
\begin{align}
C_2 & = 1, \nonumber \\
C_3 & = 81\, C_2 -54,  \nonumber\\
C_4 & = \frac{3125}{9}\, C_3 - 12500\, C_2 + 6000,  \nonumber\\
C_5 & = \frac{50421}{50}\, C_4 - \frac{3008453}{18}\, C_3 +3546277\, C_2 - 1234800, \nonumber\\
C_6 & = \frac{554769}{245}\, C_5 - \frac{140610978}{125}\, C_4 +\frac{598375026}{5}\, C_3 -\frac{8420611932}{5}\, C_2+411505920, \nonumber\\
\dots \nonumber\\
C_n & = \sum_{k=2}^{n-1}B_{n,k}C_k + F_n, \label{eq:recursion}
\end{align}
where the coefficients $B_{n,k}$ are recursively determined by:
\begin{equation}\label{eq:recursioncoeff}
	B_{n,k} = \frac{\alpha_{n,k-2}}{\alpha_{k,k-2}} - \sum_{q=1}^{n-k-1} \frac{\alpha_{n,k-2+q}}{\alpha_{k+q,k-2+q}}\, B_{k+q,k},
\end{equation}
with $A=[A_{ij}]=[\alpha_{i+2,j}]$, (indexed as $i,j\in\{0,1,2,\dots\}$), defining an infinite lower triangular matrix with nonzero entries:
\begin{equation}\label{eq:lowertriangmatrix}
	\alpha_{n,k} = (2n-3)^{2n-2}e_k(\Gamma_{n-2}),\quad n=2,3,\dots,\; k=0,1,\dots,n-2.
\end{equation}
The inhomogeneous term $F_n$ is
\begin{equation}\label{eq:Fn}
	F_n = \alpha_{n,n-2}(-1)^{n-2}(2n-3)K_{n-2} 
\end{equation}
\end{cor}
Even though a generating function for $C_n$ remains unknown, the linear relations involved in Equation \eqref{eq:SchubertCn} can be inverted to define a generating function of linear combinations of $C_n$ that does have a simple closed-form expression.
\begin{cor}\label{cor:generatingfunction}
Let $A^{-1}=[\theta_{n,k}]$ be the infinite lower triangular inverse matrix of \eqref{eq:lowertriangmatrix}. Define the generating function
\begin{equation}\label{eq:formalseries}
	Z(x) = \sum_{n=0}^{\infty}\left( \sum_{k=0}^{n} \theta_{n,k}C_{k+2} \right) x^n,
\end{equation}
then it has the closed-form formula:
\begin{equation}\label{eq:generatingfunction}
	Z(x) = \frac{1}{2x}\left(1-\frac{1}{\sqrt{1+4x}} \right).
\end{equation}
In particular,
\begin{align*}
Z(x) = & C_2 + \left(-\frac{9}{2}C_2 +\frac{1}{18}C_3\right)x + \left(  \frac{125}{6}C_2-\frac{125}{216}C_3+\frac{1}{600}C_4 \right)x^2 \\ & + \left( -\frac{72373}{720}C_2+\frac{61397}{12960}C_3-\frac{343}{12000}C_4+\frac{1}{35280}C_5 \right)x^3 \\ 
& + \left( \frac{2887727}{5600}C_2-\frac{3693673}{100800}C_3+\frac{96441}{280000}C_4-\frac{761}{1097600}C_5+\frac{1}{3265920}C_6 \right)x^4 +\dots \\
& = 1 -3x + 10x^2 -35x^3 + 126x^4 -462x^5+ 1716x^6 -6435 x^7+ 24310x^8 + \dots
\end{align*}
\end{cor}
%
%
%
\begin{rmk}
	The coefficients $\theta_{n,k}$ are computable exactly: any block of size $N\times N$ of $A$ can be written as $D(I+T)$ with $D$ diagonal and $T$ strictly lower triangular, so the inverse is $(I+\sum_{q=1}^{N-1}(-1)^{q}T^q)D^{-1}$, which is expressible in terms of $\alpha_{n,k}$. 
\end{rmk}
\begin{cor}
	Every number $C_n$ is an odd integer. 
\end{cor}
\begin{proof}
Equation \eqref{eq:SchubertCn} can be written as
$$
C_n  = (2n-3)^{2n-2} + \sum_{m=1}^{n-2} (2n-3)^{2n-2-2m}(-1)^m(2m+1)K_m e_{m}(\{(2n-3-j)j\}_{j=1}^{n-2})
$$
For any $j_l$ such that $(2n-3-j_l)j_l$ is a factor of a term in $e_{m}(\{(2n-3-j)j\}_{j=1}^{n-2})$, if $j_l$ is even that term does not contribute to the equation above, and if $j_l$ is odd then its accompanying factor $(2n-3-j_l)$ is even, and the same happens. Therefore, we conclude $C_n \equiv (2n-3)^{2n-2}\mod 2$, which is a power of an odd number, i.e., it is always odd.
\end{proof}
\subsection{Bombieri norm formula}
Our second formula is a combinatorial expansion in terms of products over certain set compositions and the counting of labelings with cycles of a certain type of blocks, which we now introduce.

\begin{defi}
An \emph{$h$-special composition} $(I,J)$ of the set $[2n-2]=\{1,2,\dots,2n-2\}$, for any integer $n\geq 2$, is a composition, i.e., an ordered partition into two disjoint subsets $I,J\subset [2n-2]$, of equal cardinality each $n-1$, where $1\in I, (2n-2)\in J$, and such that the intersection $I\cap(J-1)$ has size $h$, (with element-wise subtraction). We denote the set of such compositions as $\mathcal{P}_{2n-2}^{(h)}$.
\end{defi}

\begin{example}
Thinking of the elements of $[2n-2]$ as ordered indices, the intersection $I\cap(J-1)$ can be interpreted as those indices in $I$ exactly to the left of an index in $J$. Thus, fixing its cardinal to $h$ amounts to choosing a composition such that $I$, and $J$, consists of $h$ subsets of consecutive indices. So an $h$-special composition of $[2n-2]$ consists of choosing $n-2$ numbers from $\{2,3,\dots,2n-3\}$ such that they form $h$ groups of consecutive numbers when adding $1$ and $2n-2$ to $I$ and $J$ respectively. For instance, consider $(I,J)\in\mathcal{P}^{(3)}_{12}$, represented by following blocks:
\begin{equation*}
\ytableausetup
{boxsize=1.1em}
\ytableausetup
{aligntableaux=center}
\begin{ytableau}
*(black!10) 1 & 2 & 3 & *(black!10) 4 & *(black!10) 5 & 6 & *(black!10) 7 & *(black!10) 8 & *(black!10) 9 & 10 & 11 & 12 
\end{ytableau}
\end{equation*}
Here $I=\{1,4,5,7,8,9\}, J=\{2,3,6,10,11,12\}$ have $I\cap(J-1)=\{1,5,9\}$, thus $h=3$, and indeed each $I,J$ have three consecutive blocks of numbers: $\{1\},\{4,5\},\{7,8,9\}$ for $I$, and $\{2,3\},\{6\},\{10,11,12\}$ for $J$.
\end{example}
\begin{defi}
A \emph{$z$-block} $\BB{n}{h}$ of width $n$ and bulk size $h$ is a set of elementary boxes forming a two row diagram with $n$ boxes in each row, where $h$ of the boxes are aligned in common columns:
\begin{equation}
\BB{n}{h} = \quad
\ytableausetup
{boxsize=1.15em}
\ytableausetup
{aligntableaux=center}
\begin{ytableau}
~ & \none[\dots] & & & & \none[\dots] & & &\none &\none \\
\none~ & \none &\none & & & \none[\dots]& & & & \none[\dots] & 
\end{ytableau}
\end{equation}
A \emph{labeling} $\pi\times\rho$ on a $z$-block $\BB{n}{h}$ consists of independently assigning a permutation $\pi,\rho\in S_{n}$ to each of the rows. 
\end{defi}
\begin{defi}
A $k$-\emph{cycle} in a $z$-block labeling is a group of $k$ columns from the bulk which form a cyclic permutation when read vertically. A \emph{cyclic profile} $\underline{\lambda}=(\lambda_1,\lambda_2,\dots,\lambda_h)$ is any labeling which includes $\lambda_1$ 1-cycles, $\lambda_2$ 2-cycles, etc. in any positions in the bulk. The \emph{residual block} of a profile, $\rBB{n}{h}{\ulmbd}$, is the $z$-block obtained from the original by removing all the columns used in cycles, i.e., the $z$-block of reduced width $n-\sum_{k=1}^h k\lambda_k$ and bulk size $h-\sum_{k=1}^h k\lambda_k$.
\end{defi}
Notice that the columns of each cycle need not be consecutive. 
Let us denote by $\Lambda(n,h)$ the set of all tuples $\ulmbd$ that are feasible cycle profiles for a width $n$ and bulk size $h$.
\begin{defi}
The \emph{length} of a $z$-block $\BB{n}{h}$ for the cyclic profile $\underline{\lambda}$ is the number $L_{\underline{\lambda}}\left[\BB{n}{h}\right]$	counting the different possible labelings of the block that precisely have as many cycles of every type as specified by $\underline\lambda$ and no others.
\end{defi}
%
\begin{example}
Consider the $z$-block of width 10 and bulk size 7: 
\begin{equation*}
\BB{10}{7} = 
\ytableausetup
{boxsize=1.15em}
\ytableausetup
{aligntableaux=center}
\begin{ytableau}
~ & ~ & ~ & ~ & ~ & ~ & ~ & ~ & ~ & ~ &\none &\none \\
\none~ & \none &\none & ~ & ~ & ~ & ~ & ~ & ~ & ~ & ~ & ~ & ~
\end{ytableau}
\end{equation*}
A specific labeling is for instance

\begin{equation*}
\ytableausetup
{boxsize=1.15em}
\ytableausetup
{aligntableaux=center}
\begin{ytableau}
4 & 2 & 1 & 3 & 6  & 5 & 9 &  7 &  8 & 10 &\none &\none \\
\none~ & \none &\none & 3 &  5 &  6 &  7 &  8 &  9 & 4 & 1 & 2 & 10 
\end{ytableau}
\end{equation*}
which has 3 cycles of profile $(1,1,1,0,0,0,0)$, in a particular position and order in the bulk (here their columns are consecutive for illustrative purposes, but this need not be the case). There is 1 column with the same label in both rows, 1 pair of columns with swapped values top-bottom, and 1 set of 3 columns with a cyclic permutation when read vertically stacking the columns of the cycle one after another. In order to compute the number of labelings with this specific profile, $L_{(1,1,1)}\left[\BB{10}{7}\right]$, we proceed as follows: choose the $k$ columns that will be a $k$-cycle (a factor $\binom{h}{k}$), then choose the values to fill them (a factor $\binom{n}{k}$), order them inside the cycle (a factor $k!$), and repeat for each new cycle choosing from the remaining columns ($h-k$) and values ($n-k$), and finally multiply by the number of choices to label the boxes that do not form any cycles (which inevitably leads to recursion). Notice that ordering the cycle requires an additional factor of $2$ per cycle of size 3 or more, since the $k!$ only accounts for the sequence filling in the cycle and an additional orientation of the diagonals in the two rows needs to be chosen. In our example, the 3-cycle $(978)$ can have the columns filled as $[97][78][89]$ or as $[79][87][98]$, which corresponds to filling in/reading the cycle in the $z$-block from bottom to top instead of from top to bottom. Then
$$
L_{(1,1,1)}\left[\BB{10}{7}\right] = \binom{7}{1}\binom{10}{1}1! \binom{6}{2}\binom{9}{2}2! \binom{4}{3}\binom{7}{3}3!\,2\cdot\, L_{(0)}\left[\BB{4}{1}\right],
$$
where $L_{(0)}\left[\BB{4}{1}\right]$ counts the labelings without cycles of the residual $z$-block 
\begin{equation*}
\rBB{10}{7}{(1,1,1)} = \BB{4}{1} = 
\ytableausetup
{boxsize=1.15em}
\ytableausetup
{aligntableaux=center}
\begin{ytableau}
~ & ~ & ~ &~ &\none &\none \\
\none~ & \none &\none & ~ & ~ & ~ & ~ 
\end{ytableau}
\end{equation*}
To determine $L_{(0)}\left[\BB{4}{1}\right]$ notice that from the possible labelings $4!^2$ ($4!$ permutations at the top row per each of the $4!$ at the bottom row), we must remove the cases of a 1-cycle happening in the new bulk of size $1$, so
$$
L_{(0)}\left[\BB{4}{1}\right] = 4!^2 - L_{(1)}\left[\BB{4}{1}\right] = 4!^2 - \binom{1}{1}\binom{4}{1}1!L_{(0)}\left[\BB{3}{0}\right],
$$
and since there is no bulk in the last residual $\rBB{4}{1}{(1)}=\BB{3}{0}$, we have $L_{(0)}\left[\BB{3}{0}\right]=3!^2$, hence
$$
L_{(1,1,1)}\left[\BB{10}{7}\right] = \binom{7}{1}\binom{10}{1}1! \binom{6}{2}\binom{9}{2}2! \binom{4}{3}\binom{7}{3}3!\, 2\left[  4!^2 - \binom{1}{1}\binom{4}{1}1! 3!^2 \right] = 54867456000.
$$
\end{example}

After this example, it is easy to see that the counting of labelings without cycles on a block obeys
\begin{equation}\label{eq:zerocycles}
	L_{\underline{0}}\left[ \BB{n}{h}  \right] = n!^2 - \sum_{\substack{\underline{\lambda}\in\Lambda(n,h)\\ \underline{\lambda}\neq 0 }} L_{\underline{\lambda}}\left[ \BB{n}{h} \right],
\end{equation}
so the zero labeling of a residual block is computed in terms of the nonzero labelings of the same block, and these in terms of the succesive residuals, until reaching a $z$-block without bulk. Therefore, computing the length of $z$-blocks is an essentially recursive process.

With all these preliminaries we are ready to state our second formula that will be proved by computing the Bombieri norm of certain polynomial determinants arising in random algebraic geometry.
\begin{thm}\label{th:main}
	The number $C_n$ of lines on a generic complex projective hypersurface of degree $2n-3$ in $
	\mathbb{CP}^n$ satisfies the combinatorial expansion:
	\begin{equation}\label{eq:BombieriCn}
		\boxed{
		C_n = \sum_{h=1}^{n-1}W_{n,h}\!\!\sum_{\underline{\lambda}\in\Lambda(n-1,h)} 2^{|\underline{\lambda}|}L_{\underline{\lambda}}\left[ \BB{n-1}{h} \right], 
		}
	\end{equation}
	where the first factor is a sum over the $h$-special compositions of $[2n-2]$ given by
	\begin{align} \label{eq:coeffW}
		W_{n,h} & = \frac{1}{n!(n-1)!} \sum_{(I,J)\in\mathcal{P}_{2n-2}^{(h)}} \prod_{i\in I}(2n-2-i) \prod_{j\in J}(j-1), 
	\end{align}
	and the second factor is a weighted sum over all cycle profiles of the $z$-block lengths, which has the closed-form expression
\begin{equation}\label{eq:lengthZB}
	\sum_{\underline{\lambda}\in\Lambda(n-1,h)} 2^{|\underline{\lambda}|}L_{\underline{\lambda}}\left[ \BB{n-1}{h} \right] = \frac{n!(n-1)!}{n-h}.
\end{equation} 
\end{thm}
\begin{rmk}
Even if the weighted factors $W_{n,h}$ summing over $h$-special compositions are not known to have a more closed-form expression, their non-weighted sum does. Indeed, we have:
\begin{equation}\label{eq:Wnh}
	\sum_{h=1}^{n-1}W_{n,h} =  \frac{(2n-3)^2(2n-4)!}{n!(n-1)!}(-1)^{n-2}e_{n-2}(G_{2n-4}),
\end{equation}
where the evaluation of the elementary symmetric polynomial is at the set of numbers $G_{2n-4}=\left\{ 1-\frac{2n-3}{k} \right\}_{k=1}^{2n-4}$. To show this, let us define $\mathcal{P}_{2n-2}$ as the set of all the 2-compositions of $[2n-2]=\{1,2,\dots,2n-2\}$, i.e., the collection of ordered pairs of (possibly empty) disjoint subsets $(I,J)$ such that $I\cup J = [2n-2]$, and define $[x^{n-2}]$ as the operator extracting the coefficient of such a power of $x$ in a polynomial. Then notice that
\begin{align*}
& \sum_{h=1}^{n-1} \sum_{(I,J)\in\mathcal{P}_{2n-2}^{(h)}} \prod_{i\in I}(2n-2-i) \prod_{j\in J}(j-1) =  \sum_{\substack{(I,J)\in\mathcal{P}_{2n-2}\\ |I|=|J|=n-1\\ 1\in I,2n-2\in J }} \prod_{i\in I}(2n-2-i) \prod_{j\in J}(j-1) \\
& = (2n-3)^2 \sum_{\substack{(I,J)\in\mathcal{P}_{2n-4}\\ |I|=|J|=n-2 }} \prod_{i\in I}(2n-3-i) \prod_{j\in J}j =(2n-3)^2[x^{n-2}]\sum_{\substack{(I,J)\in\mathcal{P}_{2n-4} }} \prod_{i\in I}(2n-3-i) \prod_{j\in J}j\cdot x.
\end{align*}
It will be one of the most important steps in our proofs the fact that the sum over all 2-compositions of such a pair of products, as in the last equation, is actually the distributivity sum expansion of a product of binomial sums:
\begin{align*}
(2n-3)^2[x^{n-2}]\sum_{\substack{(I,J)\in\mathcal{P}_{2n-4} }} \prod_{i\in I}(2n-3-i) \prod_{j\in J}j\cdot x = (2n-3)^2[x^{n-2}]\prod_{k=1}^{2n-4}((2n-3-k) + kx).
\end{align*}
But then, after taking $k$ as common factor, the last product can be expanded as a polynomial in $x$ with the signed elementary symmetric polynomials as coefficients, which upon extracting the term of order $n-2$ yields Equation \eqref{eq:Wnh}.
\end{rmk}
\begin{cor}
	The asymptotic leading order of the sequence $C_n$ is upper bounded by:
	\begin{equation}
		\log C_n\leq 2n\log(n)+O(n). 
	\end{equation} 
\end{cor}
\begin{proof}
From Theorem \ref{th:main}, we may bound every term by its maximum value to obtain
\begin{align*}
	C_n & \leq \sum_{h=1}^{n-1}\frac{1}{n-1}\sum_{(I,J)\in\mathcal{P}^{(h=1)}_{2n-2}}(2n-2)^{n-1}\prod_{i\in I} \left( 1-\frac{i}{2n-2} \right)\prod_{j\in J}(j-1) \\
	& \leq (2n-2)^{n-1}\prod_{j = n}^{2n-2}(j-1) \leq (2n-2)^{n-1}(2n-3)^{n-2},
\end{align*}
that taking logarithms and expanding at leading order results in the stated bound.
\end{proof}
%
%
%
\section{Examples}\label{sec:examples}

Any of the closed-form formulas \eqref{eq:Zagier2}, \eqref{eq:libgober}, \eqref{eq:Dominici}, \eqref{eq:Harris}, \eqref{eq:SchubertCn} efficiently produce many values of the sequence in a few seconds in a modern personal computer by direct numerical computation, instead of symbolic algebra in the Chow ring:
\begin{flushleft}
\noindent
\texttt{\small
\quad 1\\
\quad 27\\
\quad 2875\\
\quad 698005\\
\quad 305093061\\
\quad 210480374951\\
\quad 210776836330775\\
\quad 289139638632755625\\
\quad 520764738758073845321\\
\quad 1192221463356102320754899\\
\quad 3381929766320534635615064019\\
\quad 11643962664020516264785825991165\\
\quad 47837786502063195088311032392578125\\
\quad 231191601420598135249236900564098773215\\
\quad 1298451577201796592589999161795264143531439\\
\quad 8386626029512440725571736265773047172289922129\\
\quad 61730844370508487817798328189038923397181280384657\\
\quad 513687287764790207960329434065844597978401438841796875\\
\quad 4798492409653834563672780605191070760393640761817269985515, ...\\}\normalsize	
\end{flushleft}
For the Bombieri norm formula \eqref{eq:BombieriCn}, instead of substituting values directly, let us explore the combinatorial structure of the expansion to derive the 27 lines on a cubic surface, and the 2875 lines on the quintic threefold, which reflects the counting of different monomials in the polynomial determinant \eqref{eq:randommatrix} from random algebraic geometry, see Section \ref{sec:Bombieri}. 
\subsection{27 lines on a cubic in $\mathbb{P}^3$}
%
In the case $n=3$, the Bombieri norm can be directly calculated by inspecting the determinant polynomial involved, which only has $7$ monomials, and this is done in \cite[Cor. 8]{Basu2019}. For $n-1=2,2n-2=4$, we have $h\in\{1,2\}$, so the $h$-special compositions needed are:
\begin{align*}
& \mathcal{P}^{(1)}_{4} = \{\;
\ytableausetup
{boxsize=1.1em}
\ytableausetup
{aligntableaux=center}
\begin{ytableau}
*(black!10) 1 & *(black!10) 2 & 3 & 4
\end{ytableau} \;\}, \quad I = \{1,2\}, J = \{3,4\}, \\
& \mathcal{P}^{(2)}_{4} = \{\;
\ytableausetup
{boxsize=1.1em}
\ytableausetup
{aligntableaux=center}
\begin{ytableau}
*(black!10) 1 & 2 & *(black!10) 3 & 4
\end{ytableau}
\;\},  \quad I = \{1,3\}, J = \{2,4\}.
\end{align*}
The factors $W_{n,h}$ are then:
\begin{align*}
W_{3,1} = \frac{1}{3!2!}(4-1)(4-2)(3-1)(4-1) = 3, \\
W_{3,2} = \frac{1}{3!2!}(4-1)(4-3)(2-1)(4-1) = \frac{3}{4}.
\end{align*}
We need the count of all the possible cycles in the $z$-blocks of width $2$ and bulk size $1$ or $2$. The feasible cycle profiles are $\Lambda(2,1)=\{(0),(1)\}$ and $\Lambda(2,2)=\{(2,0),(0,1)\}$, and so the lengths become:
\begin{align*}
& L_{(1)} \left[\;
\ytableausetup
{boxsize=1em}
\ytableausetup
{aligntableaux=center}
\begin{ytableau}
~ & \bullet &\none \\
\none~ & \bullet & ~ 
\end{ytableau} \;\right] = 
\binom{1}{1}\binom{2}{1}1! = 2, \\
& L_{(0)} \left[\;
\ytableausetup
{boxsize=1em}
\ytableausetup
{aligntableaux=center}
\begin{ytableau}
~ & &\none \\
\none~ & ~ & ~ 
\end{ytableau} \;\right] = 
2!^2 -  L_{(1)} \left[\;
\ytableausetup
{boxsize=1em}
\ytableausetup
{aligntableaux=center}
\begin{ytableau}
~ &\bullet &\none \\
\none~ & \bullet & ~ 
\end{ytableau} \;\right] = 2, \\
& L_{(2,0)} \left[\;
\ytableausetup
{boxsize=1em}
\ytableausetup
{aligntableaux=center}
\begin{ytableau}
 \bullet &\ast  \\
\bullet & \ast  
\end{ytableau} \;\right] = \binom{2}{1}1!\binom{2-1}{1}1! = 2, \\
& L_{(0,1)} \left[\;
\ytableausetup
{boxsize=1em}
\ytableausetup
{aligntableaux=center}
\begin{ytableau}
 \bullet & \ast  \\
\ast & \bullet  
\end{ytableau} \;\right] = \binom{2}{2}\binom{2}{2}2! = 2.
\end{align*}
Therefore, the number of lines on a generic cubic surface in $\mathbb{CP}^3$ is:
\begin{align*}
C_3 & =\; W_{3,1}\left( 2^1  L_{(1)} \left[\;
\ytableausetup
{boxsize=1em}
\ytableausetup
{aligntableaux=center}
\begin{ytableau}
~ & \bullet &\none \\
\none~ & \bullet & ~ 
\end{ytableau} \;\right]
+
2^0 L_{(0)} \left[\;
\ytableausetup
{boxsize=1em}
\ytableausetup
{aligntableaux=center}
\begin{ytableau}
~ & &\none \\
\none~ & ~ & ~ 
\end{ytableau} \;\right]
\right)
 \\
&\quad +W_{3,2}\left(
2^2
L_{(2,0)} \left[\;
\ytableausetup
{boxsize=1em}
\ytableausetup
{aligntableaux=center}
\begin{ytableau}
 \bullet &\ast  \\
\bullet & \ast  
\end{ytableau} \;\right]
+
2^1
 L_{(0,1)} \left[\;
\ytableausetup
{boxsize=1em}
\ytableausetup
{aligntableaux=center}
\begin{ytableau}
\bullet & \ast  \\
\ast & \bullet  
\end{ytableau} \;\right]
\right) \\
&  = 3(4+2)+\frac{3}{4}(8+4) = 27.
\end{align*}
%
%
\subsection{2875 lines on a quintic in $\mathbb{P}^4$}
%
%
In the case $n=4$, the Bombieri norm must be calculated for a polynomial determinant with 189 monomials, which is only tractable taking advantage of the combinatorial structure of Theorem \ref{th:main}.
For $n-1=3, 2n-2=6$, $h\in\{1,2,3\}$, the special $h$-compositions are:
\begin{align*}
& \mathcal{P}^{(1)}_{6} = \{\;
\ytableausetup
{boxsize=1.1em}
\ytableausetup
{aligntableaux=center}
\begin{ytableau}
*(black!10) 1 & *(black!10) 2 & *(black!10) 3 & 4 & 5 & 6
\end{ytableau} \;\},  \\
& \mathcal{P}^{(2)}_{6} = \{\;
\ytableausetup
{boxsize=1.1em}
\ytableausetup
{aligntableaux=center}
\begin{ytableau}
*(black!10) 1 & *(black!10) 2 &  3 & *(black!10) 4 & 5 & 6
\end{ytableau},
\;
\ytableausetup
{boxsize=1.1em}
\ytableausetup
{aligntableaux=center}
\begin{ytableau}
*(black!10) 1 & 2 & *(black!10) 3 & *(black!10) 4 & 5 & 6
\end{ytableau}
,
\;
\ytableausetup
{boxsize=1.1em}
\ytableausetup
{aligntableaux=center}
\begin{ytableau}
*(black!10) 1 & 2 & 3 & *(black!10) 4 & *(black!10) 5 & 6
\end{ytableau}
,
\;
\ytableausetup
{boxsize=1.1em}
\ytableausetup
{aligntableaux=center}
\begin{ytableau}
*(black!10) 1 & *(black!10)  2 & 3 & 4 & *(black!10) 5 & 6
\end{ytableau}
\;\},  \\
& \mathcal{P}^{(3)}_{6} = \{\;
\ytableausetup
{boxsize=1.1em}
\ytableausetup
{aligntableaux=center}
\begin{ytableau}
*(black!10) 1 & 2 & *(black!10) 3 & 4 &  *(black!10) 5 & 6
\end{ytableau}
\;\}, 
\end{align*}
And calculating the $W_{n,h}$ factors as in the previous example:
\begin{align*}
& W_{4,1} = \frac{1}{4!3!}(6-1)(6-2)(6-3)(4-1)(5-1)(6-1) = 25, \\
& W_{4,2} = \frac{1}{4!3!}[(6-1)(6-2)(6-4)(3-1)(5-1)(6-1) \\
&\quad\quad\quad\quad\quad + (6-1)(6-3)(6-4)(2-1)(5-1)(6-1) \\
&\quad\quad\quad\quad\quad + (6-1)(6-4)(6-5)(2-1)(3-1)(6-1) \\
&\quad\quad\quad\quad\quad + (6-1)(6-2)(6-5)(3-1)(4-1)(6-1) ] = \frac{725}{36}. \\
& W_{4,3} = \frac{1}{4!3!}(6-1)(6-3)(6-5)(2-1)(4-1)(6-1) = \frac{25}{16}.
\end{align*}
Now, the $z$-blocks can be of width $3$ and bulk size $1,2$ or $3$, and the feasible cycle profiles are then $\Lambda(3,1)=\{(0),(1)\},\; \Lambda(3,2)=\{(0,0), (1,0), (2,0), (0,1)\},\; \Lambda(3,3)=\{(3,0,0), (1,1,0), (0,0,1)\}$ (notice there is no case $(0,0,0)$, as all labelings for width 3 and bulk size 3 have cycles). The lengths are:
\begin{align*}
& L_{(1)} \left[\;
\ytableausetup
{boxsize=1em}
\ytableausetup
{aligntableaux=center}
\begin{ytableau}
~ & ~ & \bullet & \none &\none \\
\none~ & \none & \bullet & ~ & ~
\end{ytableau} \;\right] = 
\binom{1}{1}\binom{3}{1} L_{(0)} \left[\;
\ytableausetup
{boxsize=1em}
\ytableausetup
{aligntableaux=center}
\begin{ytableau}
~ & ~ & \none &\none \\
\none~ & \none & ~ & ~
\end{ytableau} \;\right]
= 3\cdot 2!^2 = 12, \\
& L_{(0)} \left[\;
\ytableausetup
{boxsize=1em}
\ytableausetup
{aligntableaux=center}
\begin{ytableau}
~ & ~ & ~ & \none &\none \\
\none~ & \none & ~ & ~ & ~
\end{ytableau} \;\right] 
= 3!^2 -  L_{(1)} \left[\;
\ytableausetup
{boxsize=1em}
\ytableausetup
{aligntableaux=center}
\begin{ytableau}
~ & ~ & \bullet & \none &\none \\
\none~ & \none & \bullet & ~ & ~
\end{ytableau} \;\right]
= 24, \\
& L_{(0,1)} \left[\;
\ytableausetup
{boxsize=1em}
\ytableausetup
{aligntableaux=center}
\begin{ytableau}
~ & \bullet & \ast &\none \\
\none~ & \ast & \bullet & ~
\end{ytableau} \;\right] 
= \binom{2}{2}\binom{3}{2}2!\cdot
L_{(0,1)} \left[\;
\ytableausetup
{boxsize=1em}
\ytableausetup
{aligntableaux=center}
\begin{ytableau}
~ & \none \\
\none~ & 
\end{ytableau} \;\right] 
= 6, \\
& L_{(2,0)} \left[\;
\ytableausetup
{boxsize=1em}
\ytableausetup
{aligntableaux=center}
\begin{ytableau}
~ & \bullet & \ast &\none \\
\none~ & \bullet & \ast & ~
\end{ytableau} \;\right] 
= 3\cdot (3-1)
= 6, \\ 
& L_{(1,0)} \left[\;
\ytableausetup
{boxsize=1em}
\ytableausetup
{aligntableaux=center}
\begin{ytableau}
~ & \bullet & &\none \\
\none~ & \bullet & & ~
\end{ytableau} \;\right] 
= \binom{2}{1}\binom{3}{1}1!\cdot
L_{(0)} \left[\;
\ytableausetup
{boxsize=1em}
\ytableausetup
{aligntableaux=center}
\begin{ytableau}
~ & ~ &  \none \\
\none~ & ~ &  
\end{ytableau} \;\right]
= 6\cdot 2 = 12
, \\ 
& L_{(0,0)} \left[\;
\ytableausetup
{boxsize=1em}
\ytableausetup
{aligntableaux=center}
\begin{ytableau}
~ & ~ & &\none \\
\none~ & ~ & & ~
\end{ytableau} \;\right] 
= 3!^2 - 12-6-6 =12, \\
& L_{(0,0,1)} \left[\;
\ytableausetup
{boxsize=1em}
\ytableausetup
{aligntableaux=center}
\begin{ytableau}
\bullet & \ast & \times \\
\ast & \times & \bullet
\end{ytableau} \text{ or }
\ytableausetup
{boxsize=1em}
\ytableausetup
{aligntableaux=center}
\begin{ytableau}
\ast & \times & \bullet \\
\bullet & \ast & \times
\end{ytableau}
 \;\right] 
= \binom{3}{3}\binom{3}{3}3!\,2 = 12, \\
& L_{(1,1,0)} \left[\;
\ytableausetup
{boxsize=1em}
\ytableausetup
{aligntableaux=center}
\begin{ytableau}
\bullet & \ast & \times \\
\bullet & \times & \ast
\end{ytableau} \;\right] 
= \binom{3}{1}\binom{3}{1}1! \binom{2}{2}\binom{2}{2}2! = 18, \\
& L_{(3,0,0)} \left[\;
\ytableausetup
{boxsize=1em}
\ytableausetup
{aligntableaux=center}
\begin{ytableau}
\bullet & \ast & \times \\
\bullet & \ast & \times
\end{ytableau} \;\right] 
= 3! = 6
\end{align*}
Therefore the number of lines on a generic quintic threefold in $\mathbb{CP}^4$ is:
\begin{align*}
C_4 & =\; W_{4,1}\left( 2^1L_{(1)} + 2^0L_{(0)} \right) \\
&\quad +  W_{4,2}\left( 2^1L_{(0,1)} + 2^2L_{(2,0)} + 2^1L_{(1,0)}+ 2^0L_{(0,0)}\right)  \\
&\quad + W_{4,3}\left( 2^1L_{(0,0,1)} + 2^2L_{(1,1,0)} + 2^3L_{(3,0,0)}  \right)  \\
& = 25\cdot 48 + \frac{725}{36}72+\frac{25}{16}144 = 2875.
\end{align*}
Notice that the diagrams represent the cycle type but the position of each cycle in the columns can vary so the ones used are just illustrative.


\section{Proof of the Schubert calculus formula}\label{sec:Schubert}

J. Harris obtains the expression \eqref{eq:Harris} in \cite{Harris79} using the splitting principle with Chern classes.
We review here this procedure deriving a similar general formula for $C_n$, Equation \eqref{eq:SchubertCn}, and establishing corollaries \ref{cor:recursion} and \ref{cor:generatingfunction}. We only use the splitting principle for the top Chern class of the $\text{Sym}^{2n-3}(S^\ast)$ bundle over the Grassmannian of projective lines $\mathbb{G}r(1,n)$, and the intersection properties of the basic Schubert cycles $\sigma_1, \sigma_{1,1}$. For all the relevant intersection theory and Schubert calculus via modern algebraic geometry see, e.g., \cite{32642016, Fulton1984}. For a review of results on the Fano scheme of lines, see \cite{ciliberto2020lines}.

By the splitting principle (see \cite[ch. 5]{32642016}) for the 2-rank bundle $S^\ast$, one can assume it is a direct sum of two line bundles, $S^\ast\cong V_1\oplus V_2$, with formal Chern classes $c_1(V_1)=\alpha,c_1(V_2)=\beta$, so by Whitney's formula its total Chern class can formally be expressed as 
$$
c(S^\ast)=c(V_1)\cdot c(V_2)=(1+\alpha)(1+\beta)=1+ (\alpha+\beta)+ \alpha\beta,
$$
where we can identify the Schubert cycles $\sigma_1 = \alpha+\beta, \sigma_{1,1}=\alpha\beta$, because of the identity $c(S^\ast) = 1 +\sigma_1+\sigma_{1,1}$, cf. \cite[Sec. 5.6.2]{32642016}.
%
%
Now, for a 2-dimensional vector space $V=V_1\oplus V_2$ we have the following decomposition:
$$
\Sym^{2n-3}(V) = \bigoplus_{k=0}^{2n-3}\left( V_1^{\otimes(2n-3-k)}\otimes V_2^{\otimes k} \right).
$$
Using Whitney's formula again and that $c_1(V_1\otimes V_2)=c_1(V_1)+c_1(V_2)$, the total Chern class of this bundle is formally
$$
 c(\Sym^{2n-3}(S^\ast)) = \prod_{k=0}^{2n-3}(1 + (2n-3-k)\alpha + k\beta),
$$
and so the top Chern class we are interested in is given by the product:
$$
 c_{2n-2}(\Sym^{2n-3}(S^\ast)) = \prod_{k=0}^{2n-3}[(2n-3-k)\alpha + k\beta] =  (2n-3)^2\alpha\beta\prod_{k=1}^{2n-4}[(2n-3-k)\alpha + k\beta].
$$
Noticing that the last product has an even number of factors that can be paired up, and that $\alpha^2+\beta^2 = \sigma_1^2-2\sigma_{1,1}$, we arrive at a product of $n-2$ pairs of factors. Using the expansion of products of binomials in terms of sums over compositions we get the following, where we recall that $\mathcal{P}_{n-2}$ is the set of 2-compositions of $[n-2]=\{1,2,\dots,n-2\}$:
\begin{align*}
 & c_{2n-2}(\Sym^{2n-3}(S^\ast)) = (2n-3)^2\alpha\beta\prod_{k=1}^{n-2}[(2n-3-k)\alpha + k\beta][k\alpha + (2n-3-k)\beta] \\
 & = (2n-3)^2\sigma_{1,1}\prod_{k=1}^{n-2}[(2n-3-k)k(\sigma_1^2-2\sigma_{1,1}) + ((2n-3-k)^2+k^2)\sigma_{1,1}] \\
 & = (2n-3)^2\sigma_{1,1}\sum_{(I,J)\in\mathcal{P}_{n-2}} \prod_{i\in I}[(2n-3-i)i(\sigma_1^2-2\sigma_{1,1})]\prod_{j\in J}[((2n-3-j)^2+j^2)\sigma_{1,1}].
\end{align*}
By collecting terms in the sum that have the same $|I|=k$, for $k=0,\dots, n-2$, we can take common factors and explicitly carry out the product of the Schubert cycles using the binomial expansion:
\begin{align*}
 & = (2n-3)^2\sigma_{1,1}\sum_{k=0}^{n-2}\sum_{\substack{(I,J)\in\mathcal{P}_{n-2}\\ |I|=k }} \left[\prod_{i\in I}(2n-3-i)i\right](\sigma_1^2-2\sigma_{1,1})^{k}\left[\prod_{j\in J}((2n-3-j)^2+j^2)\right]\sigma_{1,1}^{n-2-k} \\
 & = (2n-3)^2\sum_{k=0}^{n-2}\sum_{\substack{(I,J)\in\mathcal{P}_{n-2}\\ |I|=k }} \left[\prod_{i\in I}(2n-3-i)i\prod_{j\in J}((2n-3-j)^2+j^2)\right]\sum_{m=0}^{k}\binom{k}{m}(-2)^{k-m}\sigma_1^{2m}\sigma_{1,1}^{n-1-m}.
\end{align*}

We have to evaluate $\sigma_1^{2m}\sigma_{1,1}^{n-1-m}$ to a multiple of the top Chern class $\sigma_{n-1,n-1}$, for any values of the indices. One has that these intersection products are given by the Catalan numbers:
\begin{equation}\label{eq:CatalanSchubert}
\sigma_1^{2m}\sigma_{1,1}^{n-1-m} = \frac{1}{m+1}\binom{2m}{m}\sigma_{n-1,n-1}.
\end{equation}
This can be proved by following the discussion in \cite[p. 149]{32642016} on how the products by $\sigma_1$ work in the diagram provided in the said reference, represented for $n=4$:
\[\begin{tikzcd}
	{\sigma_0} \\
	{\sigma_1} \\
	{\sigma_2} & {\sigma_{1,1}} \\
	{\sigma_3} & {\sigma_{2,1}} \\
	{\sigma_4} & {\sigma_{3,1}} & {\sigma_{2,2}} \\
	& {\sigma_{4,1}} & {\sigma_{3,2}} \\
	&& {\sigma_{4,2}} & {\sigma_{3,3}} \\
	&&& {\sigma_{4,3}} \\
	&&&& {\sigma_{4,4}}
	\arrow[from=2-1, to=1-1]
	\arrow[from=3-2, to=2-1]
	\arrow[from=3-1, to=2-1]
	\arrow[from=4-1, to=3-1]
	\arrow[from=4-2, to=3-1]
	\arrow[from=4-2, to=3-2]
	\arrow[from=5-3, to=4-2]
	\arrow[from=5-2, to=4-2]
	\arrow[from=5-2, to=4-1]
	\arrow[from=6-3, to=5-3]
	\arrow[from=6-3, to=5-2]
	\arrow[from=7-4, to=6-3]
	\arrow[from=5-1, to=4-1]
	\arrow[from=6-2, to=5-1]
	\arrow[from=7-3, to=6-2]
	\arrow[from=7-3, to=6-3]
	\arrow[from=6-2, to=5-2]
	\arrow[from=8-4, to=7-3]
	\arrow[from=8-4, to=7-4]
	\arrow[from=9-5, to=8-4]
\end{tikzcd}\]

On the one hand, we have $\sigma_{1,1}^{n-1-m}= \sigma_{n-1-m,n-1-m}$ by \cite[Prop. 4.11]{32642016}, and on the other hand the product by $\sigma_1^{2m}$ yields the top class $\sigma_{n-1,n-1}$ by following each possible path connecting the top class (at the bottom of the diagram) with $\sigma_{n-1-m,n-1-m}$ (at one of the right corners). Therefore, one obtains the top class with a coefficient given by the counting of such paths in the diagram, which after a standard combinatorial argument (see \cite[Ex. 6.19h]{StanleyFomin1999}) is known to be the Catalan numbers.

Finally, let us simplify the sum over compositions of the products in square brackets, call it $A_{n,k}$,
introducing arbitrary variables $x,y$, such that:
$$
x^ky^{n-2-k}A_{n,k} = \sum_{\substack{(I,J)\in\mathcal{P}_{n-2}\\ |I|=k }}\prod_{i\in I}x(2n-3-i)i\prod_{j\in J}y((2n-3-j)^2+j^2).
$$
Then, writing $[x^ky^{n-2-k}]$ for the operator of extracting the coefficient of that monomial in a polynomial, we have:
\begin{align*}
A_{n,k} & = [x^ky^{n-2-k}]\sum_{\substack{(I,J)\in\mathcal{P}_{n-2}}}\prod_{i\in I}x(2n-3-i)i\prod_{j\in J}y((2n-3-j)^2+j^2) \\
& = [x^ky^{n-2-k}] \prod_{k=1}^{n-2}\left[ x(2n-3-k)k + y((2n-3-k)^2+k^2) \right] 
\end{align*}
\begin{align*}
& = [x^ky^{n-2-k}] \prod_{k=1}^{n-2}\left[ y(2n-3)^2 + (x-2y)(2n-3-k)k \right] \\
& = [x^ky^{n-2-k}]\sum_{\substack{(I,J)\in\mathcal{P}_{n-2}}}\prod_{i\in I}y(2n-3)^2\prod_{j\in J}(x-2y)(2n-3-j)j \\
& = [x^ky^{n-2-k}]\sum_{\substack{(I,J)\in\mathcal{P}_{n-2}}}y^{|I|}(2n-3)^{2|I|}(x-2y)^{|J|}\prod_{j\in J}(2n-3-j)j \\
& = [x^ky^{n-2-k}]\sum_{\substack{(I,J)\in\mathcal{P}_{n-2}}}y^{|I|}(2n-3)^{2(n-2-|J|)}\sum_{l=0}^{|J|}\binom{|J|}{l}x^l(-2y)^{|J|-l} \prod_{j\in J}(2n-3-j)j \\
& = \sum_{\substack{(I,J)\in\mathcal{P}_{n-2}\\ |J|\geq k}}(2n-3)^{2n-4-2|J|}\binom{|J|}{k}(-2)^{|J|-k} \prod_{j\in J}(2n-3-j)j \\
& = \sum_{t\geq k}^{n-2}(-2)^{t-k} (2n-3)^{2n-4-2t}\binom{t}{k} \sum_{\substack{(I,J)\in\mathcal{P}_{n-2}\\ |J| = t}} \prod_{j\in J}(2n-3-j)j,
\end{align*}
and the latter sum is precisely the elementary symmetric polynomial $e_t$ evaluated at the numbers $\{(2n-3-j)j\}_{j=1}^{n-2}$. We are left with the expression
$$
C_n = (2n-3)^{2n-2}\sum_{k=0}^{n-2}\sum_{t=k}^{n-2}\sum_{m=0}^{k}\binom{k}{m}\binom{t}{k}K_m(-2)^{t-m}e_t\left( \Gamma_{n-2} \right),
$$
where $\Gamma_{n-2}=\left\{\frac{(2n-3-j)j}{(2n-3)^2}\right\}_{j=1}^{n-2}$. Notice that the three sums can be reordered by $\sum_{k=0}^{n-2}\sum_{t=k}^{n-2}\sum_{m=0}^{k} = \sum_{t=0}^{n-2}\sum_{m=0}^t\sum_{k=m}^t $, so that we can apply the identity
\begin{equation}\label{eq:twobinomials}
\sum_{k=m}^{t}\binom{k}{m}\binom{t}{k} = 2^{t-m}\binom{t}{m},
\end{equation}
using that $\binom{k}{m}\binom{t}{k}=\binom{t}{m}\binom{t-m}{k-m}$, \cite[Eq. 5.21]{graham1994concrete}. The result is the formula
\begin{equation}\label{eq:semifinalCn}
C_n = (2n-3)^{2n-2}\sum_{t=0}^{n-2}e_{t}(\Gamma_{n-2})\sum_{m=0}^t\binom{t}{m}(-4)^{t-m}K_m.
\end{equation}
But the numbers $u_t = \sum_{m=0}^t\binom{t}{m}(-4)^{t-m}K_m$ are the generalized binomial transform of the Catalan numbers, which is of the form
$$
u_k=\sum_{m=0}^k\binom{k}{m}a^m(-c)^{k-m}b_m, \text{ with } a=1, c=4, b_m = K_m,
$$
for which the generating function $Z(x)=\sum_{k=0}^\infty u_kx^k$ has an expression in terms of the generating function $B(x)=\sum_{n=0}^\infty b_nx^n$, given by $Z(x)=\frac{1}{1+cx}B(\frac{ax}{1+cx})$, see \cite{Prodinger94}. The generating function of the Catalan numbers is known to be
$$
K(x) = \frac{1}{2x}\left( 1- \sqrt{1-4x} \right),
$$
and so the generating function $Z(x)=\sum_{k=0}^\infty u_k x^k$ is
$$
Z(x) = \frac{1}{1+4x}K\left( \frac{x}{1+4x} \right) = \frac{1}{2x}\left( 1-\frac{1}{\sqrt{1+4x}} \right).
$$ 
Therefore, $u_t$ is the coefficient of $x^{t}$ in $Z(x)$, so recalling that
$$
(1+x)^{-1/2} = \sum_{m=0}^\infty \frac{(-1)^m(2m)!}{4^m m!^2}x^m,
$$
we finally arrive at Equation \eqref{eq:SchubertCn} by extracting the coefficient from $Z(x)$ and using the definition of Catalan number in
$$
u_t = [x^t]Z(x)=[x^t]\left( \frac{1}{2x}-\sum_{m=0}^\infty \frac{(-1)^m(2m)!}{m!^22}x^{m-1} \right) = (-1)^{t+2}\frac{(2(t+1))!}{2(t+1)!^2} = (-1)^t(2t+1)K_t.
$$

In order to obtain the inhomogeneous recursion relations with variable coefficients \eqref{eq:recursion}, it is sufficient to write \eqref{eq:SchubertCn} as
$$
C_n = \sum_{k=0}^{n-2}\alpha_{n,k}u_k, \text{ with } \alpha_{n,k} = (2n-3)^{2n-2}e_k(\Gamma_{n-2}),\quad n=2,3,\dots,\; k=0,1,\dots,n-2.
$$
Then $\alpha_{n,k}$ defines the nonzero entries of an infinite lower triangular matrix $A=[A_{ij}]$, such that one can sequentially solve the values of $w_k$ in terms of $C_{k+2}$ and its predecessors:
\begin{align*}
	& C_2 = A_{00}w_0 \Rightarrow w_0 = \frac{1}{A_{00}}C_2, \\
	& C_3 = A_{10}w_0 + A_{11}w_1 = \frac{A_{10}}{A_{00}}C_2 + A_{11}w_1 \Rightarrow w_1 = \frac{1}{A_{11}}\left( C_3 - \frac{A_{10}}{A_{00}}C_2 \right) , \\
	& C_4 = A_{20}w_0 + A_{21}w_1 + A_{22}w_2 = \left( \frac{A_{20}}{A_{00}}-\frac{A_{21}A_{10}}{A_{11}A_{00}} \right)C_2 + \frac{A_{21}}{A_{11}}C_3 + A_{22}w_2, \\
 & C_5 = \frac{A_{30}}{A_{00}}C_2 + \frac{A_{31}}{A_{11}}\left( C_3 - \frac{A_{10}}{A_{00}}C_2 \right) + \frac{A_{32}}{A_{22}}\left( C_4 -\frac{A_{20}}{A_{00}}C_2 - \frac{A_{21}}{A_{11}}\left( C_3-\frac{A_{10}}{A_{00}}C_2 \right) \right) + A_{33}w_3, \\
	& = \left( \frac{A_{30}}{A_{00}}-\frac{A_{31}A_{10}}{A_{11}A_{00}} - \frac{A_{32}}{A_{22}}\left( \frac{A_{20}}{A_{00}}-\frac{A_{21}A_{10}}{A_{11}A_{00}} \right) \right)C_2 + \left( \frac{A_{31}}{A_{11}}-\frac{A_{32}A_{21}}{A_{22}A_{11}} \right) C_3 + \frac{A_{32}}{A_{22}}C_4 + A_{33}w_3\\
	& ...
\end{align*}
Collecting terms and keeping track of the indices, one arrives at Equation \eqref{eq:recursioncoeff} for the coefficient $B_{n,m}$ of $C_m, 2\leq m \leq n-1$, establishing Corollary \ref{cor:recursion}.

The diagonal of $A$ is always nonzero since it corresponds to $\alpha_{n,n-2}=(2n-3)^{2n-2}e_{n-2}(\Gamma_{n-2})$, for $n\geq 2$, thus any $N\times N$ upper-left block of $A$ is invertible, resulting in an infinite lower triangular matrix $[\theta_{n,k}]$. Inverting the relation \eqref{eq:SchubertCn} in terms of this matrix with the infinite lower triangular matrix $A^{-1}$, we obtain $u_n = \sum_{k=0}^n\theta_{n,k}C_{k+2}$. Hence, using the representation of these numbers from the generating function $Z(x)$, we have established Corollary \ref{cor:generatingfunction}.


\section{Proof of the Bombieri norm formula}\label{sec:Bombieri}

Our starting point is the following major result from random algebraic geometry by Basu-Lerario-Lundberg-Peterson \cite[Th. 7]{Basu2019}. Let $w=(w_1,\dots,w_{2n-3})\in\mathbb{C}^{2n-3}$ be a random vector with entries that are independent Gaußian variables distributed as:
$$
w_j \sim N_{\mathbb{C}}\left( 0, \binom{2n-4}{j-1} \right),\quad j =1,\dots, 2n-3.
$$
Let $w^{(1)},\dots,w^{(n-1)}$ be independent random vectors all distributed as $w$, and define the random $(2n-2)\times(2n-2)$ matrix $\hat{J}^{\mathbb{C}}_n$ as:
$$
\hat{J}^{\mathbb{C}}_n = 
\begin{bmatrix}
w^{(1)}_1 & 0 & \dots & w^{(n-1)}_1 & 0 \\
w^{(1)}_2 & w^{(1)}_1 & ~ & w^{(n-1)}_2 & w^{(n-1)}_1  \\
\vdots & w^{(1)}_2 & ~ & \vdots & w^{(n-1)}_2  \\
w^{(1)}_{2n-3} & \vdots & ~ & w^{(n-1)}_{2n-3} & \vdots  \\
0 & w^{(1)}_{2n-3} & \dots & 0 & w^{(n-1)}_{2n-3}
\end{bmatrix}
$$
\begin{thm}[\cite{Basu2019}]
	The number $C_n$ of lines on a generic hypersurface of degree $2n-3$ in $\mathbb{CP}^n$ is:
\begin{equation}\label{eq:BasuEq}
	C_n = \left( \frac{(2n-3)^{2n-2}}{\Gamma(n)\Gamma(n+1)}\prod_{k=0}^{2n-3}\binom{2n-3}{k}^{-1} \right)\mathbb{E}\vert \det\hat{J}^{\mathbb{C}}_n \vert^2.
\end{equation}
\end{thm}
The authors explain, \cite[Rmk. 6]{Basu2019}, that considering the polynomial determinant $P_n(x)$ of the matrix
\begin{equation}\label{eq:randommatrix}
A_n(x) = 
\begin{bmatrix}
x_{1,1} & 0 & \dots & x_{n-1,1} & 0 \\
\vdots & x_{1,1} & ~ & \vdots & x_{n-1,1}  \\
\binom{2n-4}{j-1}^{1/2}x_{1,j} & \vdots & ~ & \binom{2n-4}{j-1}^{1/2}x_{n-1,j} & \vdots  \\
\vdots & \binom{2n-4}{j-1}^{1/2}x_{1,j} & ~ & \vdots & \binom{2n-4}{j-1}^{1/2}x_{n-1,j}  \\
x_{1,2n-3} & \vdots & ~ & x_{n-1,2n-3} & \vdots \\ 
0 & x_{1,2n-3} & \dots & 0 & x_{n-1,2n-3}
\end{bmatrix}
\end{equation}
then the expectation value of the determinant in \eqref{eq:BasuEq} is given in terms of the Bombieri norm of $P_n(x)$:
\begin{equation}\label{eq:expectation}
	\mathbb{E}\vert \det\hat{J}^{\mathbb{C}}_n \vert^2 = (2n-2)!\; ||P_n||^2_B.
\end{equation}
Recall that, for a homogeneous polynomial of degree $D$ in $N$ variables, $P(\underline{x})=\sum_{|\underline{\alpha}|=D}P_{\underline{\alpha}}x_1^{\alpha_1}\cdots x_N^{\alpha_N}$, the Bombieri norm is
\begin{equation}\label{eq:Bnorm}
	||P(\underline{x})||_B = \sqrt{ \sum_{|\underline{\alpha}|=D}|P_{\underline{\alpha}}|^2\frac{\alpha_1!\cdots\alpha_N!}{D!} }.
\end{equation}
We shall compute explicitly the Bombieri norm of $P_n(x)$ and reveal its combinatorial structure, for arbitrary $n\geq 2$, in order to obtain the expansion formula \eqref{eq:BombieriCn} from Equation \eqref{eq:BasuEq}.
Let us first define convenient symbols to work with the matrix $A_n(x)$. Let us write $[i|j]=x_{i,j}$ for convenience of manipulating indices, and define:
\begin{equation*}
	\eta_i = \binom{2n-4}{i-1}^{1/2},
\end{equation*}
the symbols $\cancel{\delta}_2(k) = k\mod 2,\; \delta_2(k)=1-\cancel{\delta}_2(k)$, which are $1$ if $k$ is odd or even respectively, and zero otherwise, and the opposite of Kronecker's delta $\cancel{\delta}_{i,j}=1-\delta_{i,j}$. Then the entries of the matrix $A_n(x)$ can be written as
\begin{equation}\label{eq:matrixA}
	[A_n(x)]_{i,k} = \cancel{\delta}_2(k)\cancel{\delta}_{2n-2,i}\cdot \eta_i\left[\frac{k+1}{2}\;|\;i\right] + \delta_2(k)\cancel{\delta}_{1,i}\cdot\eta_{i-1}\left[\frac{k}{2}\;|\; i-1\right].
\end{equation}
Notice that the indices are well-defined integers precisely because the delta symbols introduced are only nonzero for the correct parity.
Therefore the determinant is:
\begin{equation*}\label{eq:detA}
	P_n(x) = \sum_{\sigma\in S_{2n-2}}(-1)^{\sigma}\prod_{i=1}^{2n-2} A_{i,\sigma(i)} 
	= \sum_{\sigma\in S_{2n-2}}(-1)^{\sigma}\prod_{i=1}^{2n-2} (u_i(\sigma) + w_i(\sigma)),
\end{equation*}
where
\begin{equation*}
u_i(\sigma) =  \cancel{\delta}_2(\sigma(i))\cancel{\delta}_{2n-2,i}\cdot \eta_i\left[\frac{\sigma(i)+1}{2}\;|\;i\right], \quad\quad
w_i(\sigma) = \delta_2(\sigma(i))\cancel{\delta}_{1,i}\cdot\eta_{i-1}\left[\frac{\sigma(i)}{2}\;|\; i-1\right].
\end{equation*}
Then, expanding the product as a sum over all 2-compositions $(I,J)\in\mathcal{P}_{2n-2}$ of the set $[2n-2]$, we arrive at:
$$
P_n(x) = \sum_{\sigma\in S_{2n-2}}(-1)^\sigma \sum_{(I,J)\in\mathcal{P}_{2n-2}}\prod_{i\in I}u_i(\sigma)\prod_{j\in J}w_j(\sigma),
$$
and the latter products include the product of the delta symbols, which make the terms in the sum to be nonzero if and only if
$$
\prod_{i\in I} \cancel{\delta}_2(\sigma(i))\cancel{\delta}_{2n-2,i} \prod_{j\in J}\delta_2(\sigma(j))\cancel{\delta}_{1,j} \neq 0 \Leftrightarrow 
\begin{cases} 
\sigma(i)\equiv 1\mod 2,\text{ and }i\neq 2n-2, \forall i\in I, \\
\sigma(j)\equiv 0\mod 2,\text{ and }j\neq 1, \forall j\in J.
\end{cases}
$$

Therefore, for a fixed composition $(I,J)$ the only permutations $\sigma\in S_{2n-2}$ that yield nonzero terms in the expression are those which send the indices $I$ to odd numbers and $J$ to even numbers, which can only happen for the compositions of equal length $|I|=|J|=n-1$. Equivalently, for any fixed permutation the compositions that yield nonzero products are those of length $n-1$. At the same time, the compositions must satisfy that $1\in I$ and $2n-2\in J$. So we can interchange the summations, remove the delta symbols, and write:
$$
P_n(x)=\sum_{\substack{(I,J)\in\mathcal{P}_{2n-2}\\ |I|=|J|=n-1}} \sum_{\substack{\sigma\in S_{2n-2}\\ \sigma|_I\equiv 1\!\!\!\!\mod 2\\ \sigma|_J\equiv 0\!\!\!\!\mod 2}} (-1)^\sigma\prod_{i\in I} \eta_i[\hat{\sigma}(i)\,|\, i] \prod_{j\in J} \eta_{j-1}[\tilde{\sigma}(j)\,|\, j-1],
$$
here $\hat{\sigma}(i)=\frac{\sigma(i)+1}{2},\; \tilde{\sigma}(j)=\frac{\sigma(j)}{2}$, are numbers in $[n-1]$. Then for all possible permutations of interest in $S_{2n-2}$, the output of each $\hat{\sigma},\tilde{\sigma}$ covers all possible permutations of $S_{n-1}\times S_{n-1}$.

The products of the last expression represent monomials of order $2n-2$ in the variables $x_{m,n}$ for $m=1,\dots, n-1,\; n = 1,\dots, 2n-3$. Reindexing the product in $j$ to $j-1$, and noticing that the product of the $\eta$'s is independent of the permutation, we separate the factors that have a common second index, i.e., when $i=j-1$ (which happens for the subset of indices $H=I\cap(J-1)$):
\begin{align}\label{eq:finalDet}
& P_n(x) =\sum_{\substack{(I,J)\in\mathcal{P}_{2n-2}\\ |I|=|J|=n-1}}\prod_{i\in I} \eta_i\prod_{j\in J} \eta_{j-1} \!\!\!\sum_{\substack{\sigma\in S_{2n-2}\\ \sigma|_I\equiv 1\!\!\!\!\mod 2\\ \sigma|_J\equiv 0\!\!\!\!\mod 2}}\!\!\! (-1)^\sigma\prod_{i\in I} [\hat{\sigma}(i)\,|\, i] \prod_{j\in J-1}[\tilde{\sigma}(j+1)\,|\, j] \\
& = \sum_{h=1}^{n-1} \sum_{\substack{(I,J)\in\mathcal{P}_{2n-2}\\ |I|=|J|=n-1\\ |I\cap(J-1)|=h}} \!\!\!\beta(I,J) \!\!\!\sum_{\substack{\sigma\in S_{2n-2}\\ \sigma|_I\equiv 1\!\!\!\!\mod 2\\ \sigma|_J\equiv 0\!\!\!\!\mod 2}}\!\!\! (-1)^\sigma\prod_{i\in I\backslash H} [\hat{\sigma}(i)\,|\, i] \prod_{j\in (J-1)\backslash H}[\tilde{\sigma}(j+1)\,|\, j]\prod_{k\in H} [\hat{\sigma}(k)\,|\, k][\tilde{\sigma}(k+1)\,|\, k]. \nonumber
\end{align}
Here, we have also decomposed the sum over compositions by the size of $h=|H|=|I\cap(J-1)|$, so that we finally get the $h$-special compositions of $[2n-2]$. We have defined as well
$$
\beta(I,J) = \prod_{i\in I} \eta_i\prod_{j\in J} \eta_{j-1}.
$$

With the last formula for $P_n(x)$ the $z$-blocks make their appearance. We must compute the Bombieri norm of this polynomial, and the monomials appearing in it are encoded by the three products in \eqref{eq:finalDet}. We must isolate which different monomials appear and with what coefficient, then apply Equation \eqref{eq:Bnorm}. The first two products always run over different indices by definition of composition and of $H$, so they can never produce factors of $x_{m,n}$ with higher order than 1. However, the product over $H$ can produce squares precisely when $\hat{\sigma}(k)=\tilde{\sigma}(k+1)$. This leads us to organize the factors of the three products precisely in a $z$-block, since this separates two legs without coincidences, (representing $[\hat{\sigma}(i)\,|\, i]$ and $[\tilde{\sigma}(j+1)\,|\, j]$), and a bulk of common columns with two rows (representing $[\hat{\sigma}(k)\,|\, k][\tilde{\sigma}(k+1)\,|\, k]$). For example when $|I\backslash H|=|(J-1)\backslash H|=1$:
$$
\begin{matrix}
	[\hat{\sigma}(i)\,|\, i] & [\hat{\sigma}(k_1)\,|\, k_1] \dots [\hat{\sigma}(k_h)\,|\, k_h] \\
	&  [\tilde{\sigma}(k_1+1)\,|\, k_1] \dots [\tilde{\sigma}(k_h+1)\,|\, k_h] & [\tilde{\sigma}(j+1)\,|\, j]
\end{matrix}
$$
is represented by the $z$-block:
$$
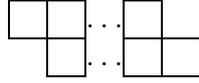
\ytableausetup
{boxsize=1.15em}
\ytableausetup
{aligntableaux=center}
\begin{ytableau}
 ~ & ~ & \none[\dots] & ~ &\none \\
\none & ~ & \none[\dots] & ~ & ~ 
\end{ytableau}
$$

Recalling the definitions of Section \ref{sec:main}, a squared variable $[m|n]^2=x^2_{m,n}$ will appear in the product over $H$ precisely when there is a 1-cycle in the corresponding $z$-block, i.e., $\hat{\sigma}(k_1)=\tilde{\sigma}(k_1+1)$. A $2$-cycle in the block corresponds to a pair of indices $k_1, k_2$ in $H$ such that $\hat{\sigma}(k_1)=\tilde{\sigma}(k_2+1)$ and $\tilde{\sigma}(k_1+1)=\hat{\sigma}(k_2)$, and similarly for longer cycles. This is crucial because every higher power than 1 in any $x^{\alpha_{m,n}}_{m,n}$ of a monomial will contribute to the Bombieri norm of that monomial with an additional factor $\alpha_{m,n}!$. Hence, if the $z$-block representing a monomial has $\lambda_1$ 1-cycles, that monomial contribution to the Bombieri norm has a factor of $2!^{\lambda_1}$. Notice that the square power is the only possible higher power for every factor of our fixed monomial because the product over $H$ is by pairs only, i.e., the $z$-block has only two rows. We have thus determined that
$$
\frac{\alpha_1!\cdots\alpha_N!}{D!} = \frac{2^{\lambda_1}}{(2n-2)!}
$$
and the denominator will cancel with the same factorial in Equation \eqref{eq:expectation}. We then need to determine the coefficient $|P_{\underline{\alpha}}|^2$ where $P_{\underline{\alpha}}$ is the number of monomials of the same indices in our overall sum.

In order to determine this coefficient, one must count how many $h$-special composition $(I,J)$ and permutations $\sigma$ with the given constraints yield the same monomial. First of all, this corresponds to permutations with fixed outputs at $\hat{\sigma}(i)$ and $\tilde{\sigma}(j+1)$ for $i\in I\backslash H, j\in(J-1)\backslash H$, since those variables of the monomial are already different for a given $(I,J)$. After a moment of reflection, one realizes the same monomial can be produced by two permutations $\sigma, \rho$ such that they interchange top and bottom in \emph{the same columns} of the bulk of the corresponding $z$-block, leaving the rest of the labeling fixed, i.e.:
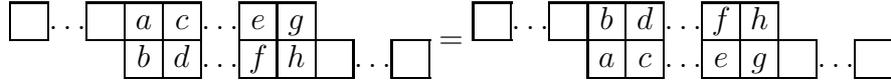
$$
\ytableausetup
{boxsize=1.15em}
\ytableausetup
{aligntableaux=center}
\begin{ytableau}
~ & \none[\dots] & ~ & a & c & \none[\dots] & e & g & \none &\none \\
\none & \none & \none & b & d & \none[\dots] & f & h & ~ & \none[\dots] & ~
\end{ytableau}
=
\ytableausetup
{boxsize=1.15em}
\ytableausetup
{aligntableaux=center}
\begin{ytableau}
~ & \none[\dots] & ~ & b & d & \none[\dots] & f & h & \none &\none \\
\none & \none & \none & a & c & \none[\dots] & e & g & ~ & \none[\dots] & ~
\end{ytableau}
$$

Let us digress for a moment to study when this happens and why we can focus on the counting disregarding the signature.
Notice the important fact that each row of the $z$-block is a permutation of $S_{n-1}$, since $\hat{\sigma}$ and $\tilde{\sigma}$ are exactly defined to output all possible permutations of $[n-1]$ when $\sigma\in S_{2n-2}$. But the values in the bulk may not coincide completely, i.e., if the bulk has size $h$, only some $t\leq h$ values from $[n-1]$ may be \emph{common values} to the top and bottom rows of the bulk, perhaps in different order. For example, in
$$
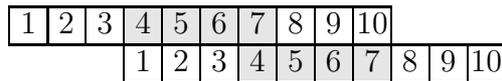
\ytableausetup
{boxsize=1.15em}
\ytableausetup
{aligntableaux=center}
\begin{ytableau}
1 & 2 & 3 & *(black!10) 4 &*(black!10) 5 &*(black!10) 6 &*(black!10) 7 & 8 & 9 & 10 &\none &\none \\
\none~ & \none &\none & 1 &  2 &  3 & *(black!10) 4 & *(black!10) 5 & *(black!10) 6 &*(black!10)7 & 8 & 9 & 10 
\end{ytableau}
$$
only $4$ values are common to both rows of the bulk, and only one column has values of those in the two rows. For $\sigma, \rho$ to swap top and bottom in some columns, as mentioned in the previous paragraph, the columns must have values from the set of common values between the two rows of the bulk. Thus, we focus on the restricted set of columns with common values in the bulk, for example:
$$
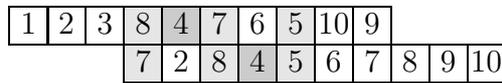
\ytableausetup
{boxsize=1.15em}
\ytableausetup
{aligntableaux=center}
\begin{ytableau}
1 & 2 & 3 & *(black!10) 8 &*(black!20) 4 & *(black!10) 7 & 6 & *(black!10) 5 & 10 & 9 &\none &\none \\
\none~ & \none &\none & *(black!10)7 &  2 & *(black!10)8 & *(black!20) 4 & *(black!10) 5 &  6 & 7 & 8 & 9 & 10 
\end{ytableau}
$$
Here, $\{5,7,8\}$ are common values that are in shared columns, whereas the value $4$ is common to both rows in the bulk but is paired column-wise with non-shared values. Hence, the same monomial can be produced by two permutations when they swap precisely the common values in the bulk that are in the same columns. 

Now, the set of common values with shared columns happens at the second index $\{k_1,\dots,k_t\}\subseteq H$, and one can isolate those columns in a rectangular block, e.g., 
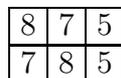
$$
\ytableausetup
{boxsize=1.15em}
\ytableausetup
{aligntableaux=center}
\begin{ytableau}
 8 &  7  &  5  \\
7 & 8  &  5
\end{ytableau}
$$
Then those columns of a $z$-block can always be thought of as a set of $t$ values out of $[n-1]$ which are ordered in the top row and perhaps permuted in the bottom row. If one thinks of the top as fixed and a relabel of $(1,2,3,\dots,t)$, one has a permutation of $S_t$ in the bottom row, which always decomposes in a product of disjoint cycles. After accounting for the permutations of the top row, these cycles correspond precisely to the $z$-block cycles as we had defined them in Section \ref{sec:main}. Therefore, for a swap of top and bottom rows to produce the same monomial, the swap can happen only if it exchanges all the columns belonging to a cycle at the same time.

It is crucial for our counting that the swapped cases do not cancel each other due to the signature $(-1)^\sigma$: indeed, to perform a swap, the common values in shared columns must change order both in the top row and bottom row, but the change requires the same number of transpositions in both rows, so the overall number of transpositions that change $\sigma$ to an equivalent permutation $\rho$ that yields the same monomial is an even number, hence having the same signature. Therefore, no cancelations happen and we just have to count cases.

Any monomial in our last expression of $P_n(x)$, for a fixed $(I,J)$ and $\sigma$, may then have cycles in its corresponding $z$-block, and any swap of each of those cycles by other permutation $\rho$, that is equal to $\sigma$ everywhere except at the swapped positions, produces the same monomial. Conversely, any monomial that is equivalent to the former must be so because of a sequence of swaps in its permutation, with the indices $I\backslash H, (J-1)\backslash H$ and $H$ fixed. Thus, to get the coefficient in the Bombieri norm, we must count first how many cycles its $z$-block has, say $\lambda_1$ 1-cycles, $\lambda_2$ 2-cycles, etc., so its cycle profile is $\underline{\lambda}=(\lambda_1,\dots,\lambda_h)$, having a total of $|\ulmbd|$ cycles. There are in consequence $2^{\lambda_2+\cdots+\lambda_h}$ possible swapped permutations yielding the same monomial, corresponding to the $2$ swap states for each cycle longer than 1. Hence, the Bombieri norm contribution of this monomial is
$$
|P_{\underline{\alpha}}|^2\frac{\alpha_1!\cdots\alpha_N!}{D!} = (\beta(I,J)2^{\lambda_2+\cdots+\lambda_h})^2 \frac{2!^{\lambda_1}}{(2n-2)!}
$$
It remains to count how many different monomials there are with this specific weight.
For fixed $h$ and $(I,J)$ this is exactly accounted for by the length of the $z$-blocks. Indeed, the length of a $z$-block $L_{\ulmbd}(\BB{n-1}{h})$, for a given cyclic profile $\ulmbd$, counts how many different labelings (which correspond to permutations in $S_{n-1}\times S_{n-1}$ for the top row and bottom row) there are which have precisely those many cycles of the types specified by $\ulmbd$. Since we are considering as equivalent any swapped cycles, we must divide by $2$ for any cycle present, and so 
$$
\frac{L_{\ulmbd}(\BB{n-1}{h})}{2^{\lambda_2+\cdots+\lambda_h}}
$$
is the number of different monomials with the same $z$-block structure and the same cyclic profile $\ulmbd$. If we sum the Bombieri norm contribution for all feasible cyclic profiles, a set which we had called $\Lambda(n-1,h)$, we are left with:
\begin{align}\label{eq:BombieriNormPmid}
||P_n(x)||^2_B & = \sum_{h=1}^{n-1} \sum_{\substack{(I,J)\in\mathcal{P}_{2n-2}\\ |I|=|J|=n-1\\ |I\cap(J-1)|=h}}\;\;\sum_{\ulmbd\in\Lambda(n-1,h)}(\beta(I,J)2^{\lambda_2+\cdots+\lambda_h})^2 \frac{2!^{\lambda_1}}{(2n-2)!}\frac{L_{\ulmbd}(\BB{n-1}{h})}{2^{\lambda_2+\cdots+\lambda_h}} \nonumber \\
& =  \sum_{h=1}^{n-1} \sum_{\substack{(I,J)\in\mathcal{P}^{(h)}_{2n-2}}} \!\!\!\frac{\beta(I,J)^2 }{(2n-2)!}\sum_{\ulmbd\in\Lambda(n-1,h)}2^{|\ulmbd|} L_{\ulmbd}(\BB{n-1}{h})\nonumber
\end{align} 

Then using Equation \eqref{eq:BasuEq}, we at last obtain
\begin{equation}\label{eq:BombieriNormPmid}
	C_n = \frac{(2n-3)^{2n-2}}{(n-1)!n!}\prod_{k=0}^{2n-3}\binom{2n-3}{k}^{-1}  \sum_{h=1}^{n-1} \sum_{\substack{(I,J)\in\mathcal{P}^{(h)}_{2n-2}}} \!\!\!\beta(I,J)^2 \!\!\! \sum_{\ulmbd\in\Lambda(n-1,h)}2^{|\ulmbd|} L_{\ulmbd}(\BB{n-1}{h}).
\end{equation}

Now we focus on the binomial coefficients of the first factor and of $\beta(I,J)$. First, let us recall that
$$
\binom{n}{k-1}=\frac{k}{n-k+1}\binom{n}{k}, \text{ for } k\geq 1.
$$
Then we have
$$
\beta(I,J)^2 = \prod_{i\in I}\binom{2n-4}{i-1}\prod_{j\in J}\binom{2n-4}{j-2} = \prod_{i\in I}\binom{2n-4}{i-1}\prod_{j\in J\backslash\{2n-2\}}\binom{2n-4}{j-1}\frac{j-1}{2n-2-j},
$$
since $I\cup J=[2n-2],I\cap J=\emptyset, 1\in I, 2n-2\in J$,
\begin{align*}
\beta(I,J)^2 & = \prod_{k=1}^{2n-3}\binom{2n-4}{k-1}\prod_{j\in J\backslash\{2n-2\}}\frac{j-1}{2n-2-j} = \prod_{k=0}^{2n-4}\binom{2n-3-1}{k}\prod_{j\in J\backslash\{2n-2\}}\frac{j-1}{2n-2-j} \\
& = \prod_{k=0}^{2n-4}\binom{2n-3}{k}\frac{2n-3-k}{2n-3}\prod_{j\in J\backslash\{2n-2\}}\frac{j-1}{2n-2-j}.
\end{align*}
Putting together this with the first constant factors of \eqref{eq:BombieriNormPmid}, we get the terms of $W_{n,h}$ of Equation \eqref{eq:coeffW}:
\begin{align*}
 \frac{(2n-3)^{2n-2}}{(n-1)!n!}\prod_{k=0}^{2n-3}\binom{2n-3}{k}^{-1} \beta(I,J)^2 & =\frac{(2n-3)}{(n-1)!n!} \prod_{k=0}^{2n-4}(2n-3-k)\!\!\!\prod_{j\in J\backslash\{2n-2\}}\frac{j-1}{2n-2-j} \\
 & = \frac{1}{(n-1)!n!} \prod_{i\in I}(2n-2-i)\prod_{j\in J}(j-1).
\end{align*}

Finally, we must find the closed-form formula for the sum over cycle profiles. For this, recall that the unsigned Stirling numbers of the first kind ${ k\brack l}$ count the number of permutations of $S_k$ that contain exactly $l$ disjoint cycles. Then split up the sum according to how many cycles there are, counted by $l_1$, and how many columns they take from the bulk of the $z$-block, counted by $k_1$:
\begin{align}
 & \sum_{\ulmbd\in\Lambda(n-1,h)}2^{|\ulmbd|} L_{\ulmbd}(\BB{n-1}{h}) =L_{\underline{0}}(\BB{n-1}{h})+ \sum_{k_1=1}^h \sum_{l_1=1}^{k_1}\; 2^{l_1}\!\!\!\sum_{\substack{\ulmbd\in\Lambda(n-1,h) \\ |\ulmbd|=l_1 \\ \sum_{j=1}^{h}j\lambda_j=k_1}} L_{\ulmbd}(\BB{n-1}{h})\nonumber \\
 & = L_{\underline{0}}(\BB{n-1}{h})+  \sum_{k_1=1}^{h}\sum_{l_1=1}^{k_1} 2^{l_1} \binom{h}{k_1}\binom{n-1}{k_1}k_1!{\; k_1\;\brack\; l_1\;}\cdot L_{\underline{0}}(\rBB{n-1}{h}{k_1} )\nonumber \\
 & = L_{\underline{0}}(\BB{n-1}{h})+  \sum_{k_1=1}^{h} \binom{h}{k_1}\binom{n-1}{k_1}k_1!^2(k_1+1)\cdot L_{\underline{0}}(\rBB{n-1}{h}{k_1} ) \label{eq:Lrecursion}
\end{align}
The reason for the second equality is that first we must choose $k_1$ columns out of $h$ that will be used in cycles, then we choose $k_1$ values out of $n-1$ to fill those columns, and then we order them in the top row. But for each of those cases in the top row, the bottom row of these columns must correspond to an ordinary permutation of $S_{k_1}$ with $l_1$ disjoint cycles, whose count is given by the Stirling number. Then, we multiply by the length of the zero profile of the residual block, $\rBB{n-1}{h}{k_1}$, since it counts how many labelings are possible in the rest of boxes, without forming any cycles, after removing $k_1$ columns. By the following properties of Stirling numbers, $ {\; k_1\;\brack\; 0\;}=0$, the third equality follows:
$$
\sum_{l_1=0}^{k_1} 2^{l_1} {\; k_1\;\brack\; l_1\;} = \sum_{l_1=0}^{k_1}\sum_{m=0}^{l_1} \binom{l_1}{m}{\; k_1\;\brack\; l_1\;} =  \sum_{m=0}^{k_1}\sum_{l_1=m}^{k_1} \binom{l_1}{m}{\; k_1\;\brack\; l_1\;} =  \sum_{m=0}^{k_1} {\; k_1+1\;\brack\; m+1\;} = (k_1+1)!.
$$
Similarly,
$$
\sum_{l=0}^{k} {\; k\;\brack\; l\;} = k!.
$$
The zero cycle length is related to the other terms by Equation \eqref{eq:zerocycles}, so we obtain:
\begin{align}
 & \sum_{\ulmbd\in\Lambda(n-1,h)}2^{|\ulmbd|} L_{\ulmbd}(\BB{n-1}{h}) =(n-1)!^2+  \sum_{k_1=1}^{h} \binom{h}{k_1}\binom{n-1}{k_1}k_1!^2(k_1)\cdot L_{\underline{0}}(\rBB{n-1}{h}{k_1} ) \label{eq:Lrecursion}
\end{align}
The length of the residual block with zero cycles has a recursive expansion given by Equation \eqref{eq:zerocycles} again, which using the previous reasoning is
\begin{align*}
& L_{\underline{0}}\left[ \rBB{n-1}{h}{k_1}  \right] = (n-1-k_1)!^2\; -\!\!\! \sum_{\substack{\underline{\mu}\in\Lambda(n-1-k_1,h-k_1)\\ \underline{\mu}\neq 0 }} L_{\underline{\mu}}\left[ \BB{n-1-k_1}{h-k_1} \right] \\
& = (n-1-k_1)!^2\; - \sum_{k_2=1}^{h-k_1} \binom{h-k_1}{k_2}\binom{n-1-k_1}{k_2}k_2!^2\cdot L_{\underline{0}}(\rBB{n-1-k_1}{h-k_1}{k_2} )
\end{align*}
The recursion is involved but can be kept track of by defining:
\begin{equation}
	a_i = (n-1-\sum_{j=1}^i k_j)!^2,\quad\quad b_i = \binom{h-\sum_{j=1}^{i-1} k_j}{k_i}\binom{n-1-\sum_{j=1}^{i-1} k_j}{k_i}k_i!^2,
\end{equation}
so that the factors in \eqref{eq:Lrecursion} expand as
\begin{align}
& \sum_{\ulmbd\in\Lambda(n-1,h)}2^{|\ulmbd|} L_{\ulmbd}(\BB{n-1}{h}) = (n-1)!^2+ \sum_{k_1=1}^{h} k_1b_1\left[ a_1 - \sum_{k_2=1}^{h-k_1} b_2 \left[ a_2 -  \sum_{k_3=1}^{h-k_1-k_2} b_3\cdots  \right] \right] \nonumber \\
& = (n-1)!^2+ \sum_{k_1=1}^{h} k_1\left[ a_1b_1 - \sum_{k_2=1}^{h-k_1} a_2b_1b_2 + \sum_{k_2=1}^{h-k_1}\sum_{k_3=1}^{h-k_1-k_2} a_3b_1b_2b_3 -\cdots \right] \nonumber \\
& = (n-1)!^2+ \sum_{k_1=1}^{h} k_1\left[ a_1b_1 +\sum_{p=2}^h\sum_{k_2=1}^{h-k_1}\cdots\!\!\!\sum_{k_p=1}^{h-\sum_{j=1}^{p-1}k_j}(-1)^{p+1} a_p\prod_{j=1}^{p}b_j \right], \label{eq:lastsum}
\end{align}
because the process terminates when all the columns of the bulk are removed, so the recursion depth is $h$. Now, the factorials in the product telescope, simplifying to
$$
a_p\prod_{j=1}^{p}b_j =h!(n-1)!\frac{(n-1-\sum_{j=1}^p k_j)!}{(h-\sum_{j=1}^p k_j)!}.
$$
In equality \eqref{eq:lastsum} for every $p$ there are $p-1$ sums over $k_2,\dots,k_p$ for which only the total value $t=\sum_{j=2}^p k_j$ is of relevance. We notice that only one sum over the possible values of $t=p-1,\dots,h-k_1$ needs to be carried out as long as we multiply by the number of compositions of $t$ in exactly $p-1$ parts, which accounts for the different values of $k_2,\dots,k_p$ yielding the same $t$. This number of compositions is $\binom{t-1}{p-2}$. Overall, the final result is
\begin{align*}
& \sum_{\ulmbd\in\Lambda(n-1,h)}2^{|\ulmbd|} L_{\ulmbd}(\BB{n-1}{h}) = (n-1)!^2 +\sum_{k=1}^{h} k\left[ (n-1-k)!^2 \binom{h}{k}\binom{n-1}{k}k!^2 + \right. \\
& \hspace{5cm}\left. + h!(n-1)!\sum_{p=2}^h\sum_{t=p-1}^{h-k} (-1)^{p+1}\binom{t-1}{p-2} \frac{(n-1-k-t)!}{(h-k-t)!} \right].
\end{align*}
This long expression massively simplifies to Equation \eqref{eq:lengthZB}. Taking common factors it can be written in terms of only binomial coefficients as
\begin{align*}
\sum_{\underline{\lambda}\in\Lambda(n-1,h)} 2^{|\underline{\lambda}|}L_{\underline{\lambda}}\left[ \BB{n-1}{h} \right] & = (n-1)!^2\left[ 1+\sum_{k=1}^h k\binom{h}{k}\binom{n-1}{k}^{-1}  \right. \\ 
& \hspace{0cm}\times\left.\left[ 1+ \sum_{p=2}^h\sum_{t=p-1}^{h-k}(-1)^{p+1}\binom{t-1}{p-2}\binom{h-k}{t}\binom{n-1-k}{t}^{-1}
		\right] \right]. \nonumber
\end{align*}
By exchanging the two sums, the inner square brackets turn out to be:
$$
 1+ \sum_{t=1}^{h-k}\sum_{p=0}^{\min(t-1,h-2)}(-1)^{p+1}\binom{t-1}{p}\binom{h-k}{t}\binom{n-1-k}{t}^{-1} = \begin{cases}
 	1, \text{ if } k=h, \\ \frac{n-1-h}{n-1-k}, \text{ if } 1\leq k\leq h-1,
 \end{cases}
$$
because the alternating sum over $p$ cancels except for a term when $h-k\leq h-2$, (here we specifically separate the cases to avoid division by zero when $h=n-1, k=h$, but assume the value $1$ in that case in what follows). Now, recall the binomial coefficient properties $\binom{m}{k}=\frac{m-k+1}{k}\binom{m}{k-1}$ and $\binom{h-1}{k-1}=\binom{h}{k}-\binom{h-1}{k}$, so that we can rewrite everything to apply the following result from \cite[Eq. 5.33]{graham1994concrete}
\begin{equation}
	\sum_{j=0}^{m}\binom{m}{j}\binom{r}{j}^{-1}=\frac{r+1}{r+1-m},\text{ for } m\geq 0\text{ and }r\notin\{0,1,\dots,m-1\}.
\end{equation}
With all this one obtains the final result:
\begin{align*}
& \sum_{\underline{\lambda}\in\Lambda(n-1,h)} 2^{|\underline{\lambda}|}L_{\underline{\lambda}}\left[ \BB{n-1}{h} \right] = (n-1)!^2\left[ 1+\sum_{k=1}^h k\binom{h}{k}\binom{n-1}{k}^{-1} \frac{n-1-h}{n-1-k} \right]  \\
& = (n-1)!^2\left[ 1 + \sum_{k=1}^{h}\frac{k}{k+1}\binom{h}{k}\binom{n-1}{n+1}^{-1}(n-1-h) \right] \\
& = (n-1)!^2\left[ 1 + \frac{h(n-1-h)}{n-1}\sum_{k=1}^h\binom{h-1}{k-1}\binom{n-2}{k}^{-1} \right] \\
& = (n-1)!^2\left[ 1 + \frac{h(n-1-h)}{n-1}\left[ \sum_{k=0}^h\binom{h}{k}\binom{n-2}{k}^{-1}- \sum_{k=0}^{h-1}\binom{h-1}{k}\binom{n-2}{k}^{-1}\right] \right] \\
& =  (n-1)!^2\left( 1 +\frac{h}{n-h} \right).
\end{align*}


\section{Concluding remarks}

In the present work we have established two formulas for the number of lines on a generic hypersurface of degree $2n-3$ in complex projective space $\mathbb{CP}^n$. The first is a closed-form formula in terms of the Catalan numbers and elementary symmetric polynomials very similar to Harris' expression, as it is derived from Schubert calculus using the splitting principle with Chern classes. This leads to a linear difference recurrence expressing $C_n$ as an unbounded non-homogeneous recursion relation with variable coefficients, which in turn yields a simple generating function for the sequence of certain linear combinations of the $C_n$. It remains unknown whether a closed-form formula for the generating formal power series of the sequence $C_n$ itself exists at all. Our first formula, or Harris', may be used to study congruence relations that were not possible using the van der Waerden-Zagier formula by taking advantage of its relation to the Catalan numbers instead of the Stirling numbers. 
On the other hand, the second result of the paper is a combinatorial expansion for $C_n$ in terms of certain set compositions and counting of cycle labelings in blocks. These account for the monomials in the computation of the Bombieri norm of certain polynomial determinants arising from random algebraic geometry, where a major result by Basu-Lerario-Lundberg-Peterson relates this to the number of lines on generic hypersurfaces. As examples, we have computed the 27 lines on a cubic surface and 2875 lines on a quintic threefold by exploring the combinatorics of the Bombieri norms of these cases. As applications, we have reobtained the odd parity of the sequence and a leading order asymptotic upper bound. Future lines of work include investigating whether it is possible to obtain non-variable recursion relations or generating functions of the $C_n$ themselves, or of the factors in the Bombieri norm formula, and studying congruence properties of the sequence taking advantage of the Catalan numbers properties. It would be also of much interest to extend any of the methods presented here to the real case, where the expectation value of the polynomial determinant in the BLLP formula is not interpreted as a Bombieri norm. Finally, a general formula can be obtained for the number of lines on a generic complete intersection of arbitrary fixed codimension, see Appendix \ref{sec:append}, so in future work we would like to establish congruence relations and asymptotic expansions for this general case.



\appendix
\section{The number of lines on complete intersections}\label{sec:append}

A. S. Libgober obtains in \cite[Th. 3]{Libgober1973} a formula for the number of lines on generic complete intersections of hypersurfaces in $\mathbb{CP}^{n}$ for arbitrary fixed codimension, when the corresponding degrees and ambient dimension match to yield a finite solution. His method consists of representing the class of the Fano variety of lines as a combination of Schubert cycles without using Chern classes. In this appendix, we extend the Schubert calculus computation from Section \ref{sec:Schubert}, using the splitting principle for a Chern class representation of the Fano variety, to obtain an alternative formula.

By \cite[Prop. 6.4]{32642016}, we have that the class, in the Chow ring of the corresponding Grassmannian, of the Fano scheme $F_k(X)$ of $k$-planes on a hypersurface $X\subset\mathbb{CP}^n$ of degree $d$ is
\begin{equation}
	[F_k(X)] = c_{\binom{k+d}{k}}(\Sym^d S^\ast) \in A(\mathbb{G}r(k,n)).
\end{equation}
To count the number of lines on a generic complete intersection $X=\bigcap_{i=1}^k Z_i\subset\mathbb{CP}^n$, it is sufficient to find the classes $F_1(Z_i)$ of the lines on each hypersurface $Z_i$ and multiply them in $A(\mathbb{G}r(1,n))$, i.e., we need
\begin{equation}\label{eq:productChern}
 \prod_{i=1}^k [F_1(Z_i)] = \prod_{i=1}^k c_{d_i+1}(\Sym^{d_i} S^\ast).
\end{equation}
Since the dimension of the Grassmannian is $2(n-1)$, this will yield a finite number of lines on a complete intersection when
\begin{equation}
	2(n-1) = \sum_{i=1}^k (d_i+1),
\end{equation}
which restricts the possible tuples of degrees for a fixed codimension.

The computation of $c_{d_i+1}(\Sym^{d_i} S^\ast)$ is exactly the same as in Section \ref{sec:Schubert} just taking into account some extra factors when the degree $d_i$ is even. Let us denote by $\{\cdot\}$ terms that must be present only in that case, i.e., when $d_i\equiv 0\mod 2$, and define $D_i=\lfloor \frac{d_i-1}{2}\rfloor$:
\begin{align*}
& c_{d_i+1}(\Sym^{d_i} S^\ast) = \prod_{j=0}^{d_i}[(d_i-j)\alpha + j\beta] = d_i^2\alpha\beta\prod_{j=1}^{d_i-1}[(d_i-j)\alpha + j\beta] \\
& = d_i^2\alpha\beta\left\{\frac{d_i}{2}(\alpha+\beta)\right\} \prod_{j=1}^{D_i}[(d_i-j)\alpha +j\beta][j\alpha +(d_i-j)\beta]
\end{align*}
\begin{align*}
& = d_i^2\sigma_{1,1}\left\{\frac{d_i}{2}\sigma_1\right\} \prod_{j=1}^{D_i}[ (d_i-j)j(\sigma_1^2-2\sigma_{1,1}) + ((d_i-j)^2+j^2)\sigma_{1,1} ] \\
& = d_i^2\sigma_{1,1}\left\{\frac{d_i}{2}\sigma_1\right\} \sum_{l=0}^{D_i}\sum_{\substack{(I,J)\in\mathcal{P}_{D_i}\\ |I|=l }}\prod_{r\in I}(d_i-r)r\prod_{s\in J}((d_i-s)^2+s^2)\sum_{m=0}^l \binom{l}{m}(-2)^{l-m}\sigma_1^{2m}\sigma_{1,1}^{D_i-m}
\end{align*}
Applying the same simplification as in Section \ref{sec:Schubert} to the factors summing over compositions, reordering the sums, and applying Equation \eqref{eq:twobinomials}, we arrive at
\begin{align*}
c_{d_i+1}(\Sym^{d_i} S^\ast) & = d_i^2\left\{\frac{d_i}{2}\right\}\sum_{l=0}^{D_i}\sum_{t=l}^{D_i}d_i^{2D_i-2t}\binom{t}{l}e_t(\Xi(d_i))\sum_{m=0}^l \binom{l}{m}(-2)^{t-m}\sigma_1^{2m+\{1\}}\sigma_{1,1}^{D_i-m+1} \\
& = d_i^{2D_i+2}\left\{\frac{d_i}{2}\right\} \sum_{t=0}^{D_i} d_i^{-2t}e_t(\Xi(d_i))\sum_{m=0}^t\binom{t}{m}(-4)^{t-m}\sigma_1^{2m+\{1\}}\sigma_{1,1}^{D_i-m+1}
\end{align*}
where we define the set of integers $\Xi (d_i)=\{(d_i-j)j\}_{j=1}^{D_i}$.

In Equation \eqref{eq:productChern}, the final product of the Schubert cycles, multiplying all factors $\sigma_1^{2m_i+\{1\}}\sigma_{1,1}^{D_i-m_i+1}$, will yield the top Schubert cycle class with a factor given by the Catalan number $K_{M(\underline{m})}$, with $M(\underline{m})=\sum_{i=1}^km_i+\{1\}$, by the same type of argument that established Equation \eqref{eq:CatalanSchubert}. Here $\{1\}$ is understood to be $1$ if and only if any of the $d_i\equiv 0\mod 2$, so that it is $0$ if and only if all the $d_i$ are odd integers. With this notation, we obtain that the expected number of lines, $C(\underline{d})$, on a complete intersection of hypersurfaces of degrees $\underline{d}=(d_1,\dots,d_k)$ in $\mathbb{CP}^N$, with $N=1+\frac{1}{2}\sum_{i=1}^k(d_i+1)$, is
\begin{equation*}
\boxed{
	C(\underline{d}) = \left[ \prod_{i=1}^k d_i^{2D_i+2}\left\{\frac{d_i}{2}\right\} \right]\sum_{t_1=0}^{D_1}\sum_{m_1=0}^{t_1}\cdots\sum_{t_k=0}^{D_k}\sum_{m_k=0}^{t_k} \left[ \prod_{i=1}^k d_i^{-2t_i}e_{t_i}(\Xi(d_i))\binom{t_i}{m_i}(-4)^{t_i-m_i} \right] K_{M(\underline{m})}
}
\end{equation*}

In Table \ref{tab:my-table} we give the first numbers of lines on a complete intersection of codimension $2$.

\Small
\begin{table}[]
\begin{tabular}{@{}l|lllllllll@{}}
  & 1 & 2  & 3    & 4      & 5        & 6          & 7             & 8               & 9                  \\ \midrule
1 & 1 &    & 27   &        & 2875     &            & 698005        &                 & 305093061          \\
2 &   & 16 &      & 1280   &          & 241920     &               & 86073344        &                    \\
3 &   &    & 1053 &        & 136125   &            & 37022391      &                 & 17425851975        \\
4 &   &    &      & 111616 &          & 22546944   &               & 8420392960      &                    \\
5 &   &    &      &        & 19188125 &            & 5503443575    &                 & 2688242739975      \\
6 &   &    &      &        &          & 4782986496 &               & 1850898677760   &                    \\
7 &   &    &      &        &          &            & 1636365422153 &                 & 820972669886649    \\
8 &   &    &      &        &          &            &               & 735546407124992 &                    \\
9 &   &    &      &        &          &            &               &                 & 420472391422517289
\end{tabular}
\caption{Number of lines on a codimension $2$ complete intersection in $\mathbb{CP}^{2+(d_1+d_2)/2}$. Every pair of row and column fixes the degrees $(d_1,d_2)$ of the hypersurfaces. (Note the table is a symmetric matrix but only the upper triangular part is represented).}
\label{tab:my-table}
\end{table}
\normalsize



\printbibliography

\end{document}